\documentclass[11pt,letterpaper,reqno]{amsart}
\usepackage[T1]{fontenc}
\usepackage[utf8]{inputenc}
\usepackage[margin=1in]{geometry}
\usepackage{amssymb}
\usepackage{parskip}

\newtheorem{theorem}{Theorem}[section]
\newtheorem{lemma}[theorem]{Lemma}
\newtheorem{prop}[theorem]{Proposition}

\newcommand{\la}{\langle}
\newcommand{\ra}{\rangle}
\newcommand{\li}{\left}
\newcommand{\ri}{\right}
\newcommand{\R}{\mathbb{R}}
\newcommand{\s}{\mathcal{S}}
\newcommand{\C}{\mathbb{C}}

\newcommand{\N}{\mathbb{N}}
\newcommand{\re}{\operatorname{Re}}
\newcommand{\im}{\operatorname{Im}}
\let\epsilon\varepsilon
\usepackage{xcolor}

\definecolor{darkred}{rgb}{0.93,0.0,0.0}
\definecolor{mygreen}{rgb}{0.0,0.65,0.0}
\usepackage[
  colorlinks=true,
  linkcolor=darkred,     
  urlcolor=darkred,      
  citecolor=mygreen, 
  filecolor=darkred
]{hyperref}

\title[Harmonic Approximation and Resolvent Estimates]{Harmonic Approximation and Resolvent Estimates for Semiclassical Non-Self-Adjoint Operators}
\author{Stepan Malkov}
\address{Department of Mathematics, University of California, Los Angeles, CA 90095, USA.}
\email{malkov@math.ucla.edu}
\date{\today}
\numberwithin{equation}{section}
\begin{document}
\begin{abstract}
    We study resolvent estimates and bounds on the low lying spectrum for a broad class of non-self-adjoint non-elliptic $h$--pseudodifferential operators with critical points. Imposing dynamical conditions on the average of the real part of the principal symbol along the Hamilton flow of the imaginary part, we establish precise semiclassical resolvent estimates in an $O(h)$--neighborhood of the boundary of the semiclassical pseudospectrum, away from the eigenvalues of quantizations of the quadratic approximations of the principal symbols of the operators.
\end{abstract}
\maketitle

\tableofcontents
\section{Introduction and statement of results}
The study of non-self-adjoint operators has a long tradition in mathematical physics, including areas such as scattering theory \cite{zworski2017scattering}, \cite{sjostrand_resonance_lectures}, kinetic equations \cite{HerauNier2004}, \cite{herau_sjostrand_stolk_kfp}, and $\mathcal{PT}$-symmetric quantum mechanics \cite{BenderHook2024}, \cite{CalicetiGraffiHitrikSjostrand2012}. A major difficulty in the non-self-adjoint spectral theory is that the norm of the resolvent may be very large even far from the spectrum. Thus, while we have
\begin{equation}
    \|(P-z)^{-1}\|_{\mathcal{L}(H,H)} = \frac{1}{\text{dist}(z,\text{Spec}(P))}, \quad z \in \C \setminus \text{Spec}(P),
\end{equation}
for a self-adjoint operator $P$ acting on a complex Hilbert space $H,$ in contrast, the resolvent norm can be much larger than $\text{dist}(z,\text{Spec}(P))^{-1}$ when the operator is non-self-adjoint \cite{TrefethenEmbree2005}, \cite{Sjostrand2019}. Following \cite{TrefethenEmbree2005}, we may therefore introduce the notion of the pseudospectrum, defined roughly as the region in the complex spectral plane where the resolvent norm is large. When making this notion precise, say, when working on $\R^n,$ one needs to choose a scaling for the operators, and the usual such scaling, which is used throughout this work, is the semiclassical one \cite{Zworski2012}, where we have $P=P(x,hD;h).$ Here, $x \in \R^n$ and $0<h \ll 1$ is the semiclassical parameter. Letting $p$ be the semiclassical principal symbol of $P,$ we may introduce the semiclassical pseudospectrum of $P,$
\begin{equation}
    \Sigma = \overline{p(\R^{2n})} \subseteq \C.
\end{equation}
Under some natural assumptions on $P,$ including an ellipticity condition at infinity, we know, as a consequence of the calculus of $h$--pseudodifferential operators \cite[Chapter 4]{Zworski2012}, that the resolvent $(P-z)^{-1}$ exists and is uniformly bounded on $L^2(\R^n)$ as $h \to 0^+,$ for $z \not \in \Sigma.$ On the other hand, the existence of the so-called H\"{o}rmander-Davies quasimodes \cite{Hormander1960,Davies1999,Zworski2001,DenckerSjostrandZworski2004} indicates that when the spectral parameter $z \in \C$ varies in the interior of $\Sigma,$ the norm of the resolvent of $P$ cannot be expected to be bounded by a fixed negative power of $h$ as $h \to 0^+.$ 

It then becomes natural to study semiclassical resolvent estimates and the spectrum when the spectral parameter $z$ is confined to a small $h$--dependent neighborhood of the boundary of the semiclassical pseudospectrum $\Sigma.$ From the point of view of time-dependent evolution problems, a particularly interesting case occurs when $\re p 
\geq 0$ and the corresponding operator $P$ admits a maximally accretive realization on $L^2(\R^n).$ Associated to $P$ is then the evolution semigroup $e^{-\frac{tP}{h}}, t \geq 0,$ and when examining the link between the large time behavior of the semigroup and the spectrum of $P$ near the imaginary axis, the most significant spectral region is the one given by $\{z \in \C: 0 \leq \re z \leq O(h)\}.$

The present work is motivated by the considerations above, and more specifically, by the series of papers \cite{hitrik_starov_i}, \cite{hitrik_starov_ii}, where resolvent estimates and spectral asymptotics in $O(h)$--neighborhoods of the origin were studied for a broad class of non-self-adjoint $h$--pseudodifferential operators $P$ on $\R^n,$ with principal symbols satisfying $\re p \geq 0$ and with $(\re p)^{-1}(0)$ finite. Assuming furthermore that $\im p$ vanishes to the second order along $(\re p)^{-1}(0)$ and that the corresponding quadratic approximations of $p$ enjoy some partial ellipticity properties, complete asymptotic expansions for the eigenvalues of $P$ in discs of the form $D(0,Ch)$ were established and precise resolvent estimates in such discs, away from the eigenvalues, were obtained in \cite{hitrik_starov_i}, \cite{hitrik_starov_ii}. In this work, we shall consider a related class of non-self-adjoint semiclassical operators with critical points, relaxing some of the assumptions of \cite{hitrik_starov_i}, \cite{hitrik_starov_ii} and replacing them by more natural ones of dynamical nature. Crucially, we will no longer require that the zero set $(\re p)^{-1}(0)$ should be finite. Instead, a finiteness assumption will be made concerning the set of critical points of $p$ with purely imaginary critical values, 
\begin{equation}
    \mathcal{C} = \{X \in \R^{2n}: \re p(X)=0, \,\, H_{\im p}(X) =0\}.
\end{equation}
Here, $H_{\im p}$ is the Hamilton vector field of $\im p.$ Following the general ideas of the method of averaging, introduced and developed in \cite{HerauHitrikSjostrand2008} in the context of second order differential operators of Kramers-Fokker-Planck type, we shall assume that the time average
\begin{equation}\label{eq:avg_defn}
    \la \re p \ra_{\im p,T}(X) := \frac{1}{2T} \int_{-T}^T \re p(e^{tH_{\im p}}X)dt, \quad X \in \R^{2n},
\end{equation} 
satisfies 
\begin{equation}
    \la \re p\ra_{\im p,T}(X) \sim |X|^2, \quad X \to 0
\end{equation}
for some $T>0,$ and away from the origin one roughly has $\la \re p\ra_{\im p,T} \geq \frac1C.$ Here, we assume for simplicity that $\mathcal{C}=\{0\}.$ Under the dynamical conditions above, we shall obtain accurate resolvent estimates in a region of the form $\{z \in \C: 0 \le \re z \leq Ch\}$ and establish bounds on the spectrum of $P$ in this region in terms of the eigenvalues of the quadratic approximations. Let us now describe the assumptions and state the results demonstrated in this paper.

Throughout, we work on $\R^n$ and let $X=(x,\xi) \in \R^{2n}$ denote the phase space variables associated to the phase space $T^*\R^n \cong \R^{2n}.$ 
 
Let $p_0 =p_0(x,\xi) \in C^\infty(\R^{2n};\C)$ be such that \begin{equation}\label{ass_pos}
    \re p_0 (X)\geq 0, \quad X \in \R^{2n}. 
\end{equation}
We shall impose a quadratic growth restriction on $p_0$ by requiring that 
\begin{equation}\label{ass_der}
    \partial^\alpha p_0  \in L^\infty(\R^{2n}), \quad |\alpha| \geq 2,
\end{equation}
and assume that for some $C>0,$ one has
\begin{equation}\label{ass_symbol_class}
|\im p_0(X)| \leq C(1+\re p_0(X)), \quad X \in \R^{2n}.
\end{equation}
As shown in Section \ref{order_function}, (\ref{ass_pos}) and (\ref{ass_der}) imply that the function 
\begin{equation}
    m(X):=1+\re p_0(X)
\end{equation}
is an order function, in the sense of \cite[Chapter~7]{Dimassi_Sjostrand_1999}. Associated to the order function $m$ is the symbol class
\begin{equation}
 S(m) :=\{a \in C^\infty(\R^{2n}):\forall \alpha \in \N^{2n}, \exists C_\alpha>0, |\partial^\alpha a|\leq C_\alpha m\},  
\end{equation} 
with (\ref{ass_symbol_class}) implying that $p_0 \in S(m)$ (see Lemma \ref{order_function_lemma}). 

We now let $p:=p(x,\xi;h) \in S(m)$ be a complex-valued $h$--dependent symbol that admits an asymptotic expansion of the form
\begin{equation}\label{ass_full_symbol}
p(x,\xi;h) \sim \sum_{k=0}^\infty h^k p_k(x,\xi)
\end{equation}
in $S(m)$ for all $h>0$ small enough, in the sense that $p_k \in S(m), k \geq 0,$ are symbols independent of $h$ and for each $N \in \N,$ one has
\begin{equation}
    p-\sum_{k=0}^N h^k p_k \in h^{N+1} S(m). 
\end{equation}
The leading term $p_0$ in this expansion is the semiclassical principal symbol of $p,$ whereas the sub-leading term $p_1$ is the semiclassical subprincipal symbol of $p.$

We shall assume that the real part of the principal symbol $p_0$ is elliptic at infinity, in the sense that \begin{equation}\label{ass_elliptic}
\exists C>0, \quad  |X|\geq C \implies  \re p_0(X) \geq \frac{1}{C},
\end{equation}
and note that (\ref{ass_elliptic}) is equivalent to the ellipticity at infinity condition 
\begin{equation}
    \exists C>0, \quad |X| \geq C \implies \re p_0(X) \geq \frac{m(X)}{C},
\end{equation}
in the symbol class $S(m).$

Our main hypothesis will be that the critical set
\begin{equation}
    \mathcal{C}:=\{X \in \R^{2n}: \re p_0(X)=0, H_{\im p_0}(X)=0\}
\end{equation}
consists of a single point, which we assume to be the origin,
\begin{equation}\label{ass_morse}
    \mathcal{C}=\{0\}.
\end{equation}
Here, $H_b:= \nabla_\xi b \cdot \partial_x-\nabla_x b \cdot \partial_\xi$ is the Hamilton vector field associated to a symbol $b \in C^\infty(\R^{2n};\R).$ As the assumptions (\ref{ass_pos}) and (\ref{ass_morse}) imply that $dp_0(0)=0,$ the Taylor expansion of $p_0$ at the origin takes the form  
\begin{equation}\label{ass_quad}
p_0(X)=p_0(0)+q(X)+O(|X|^3), \quad X \to 0.
\end{equation}
Here, $p_0(0) \in i \R$ and $q$ is the quadratic approximation to $p_0$ at $0.$ It follows from (\ref{ass_pos}) that $\re q$ is a positive semi-definite quadratic form on $\R^{2n}.$

Finally, following \cite{HerauHitrikSjostrand2008}, we impose two dynamical conditions on the averages of $\re p_0$ and $\re q$ along the flows of $H_{\im p_0}$ and $H_{\im q},$ respectively. Defining the averages $\la \re p_0\ra_{\im p_0,T}$ and $\la \re q\ra_{\im q,T}$ as in (\ref{eq:avg_defn}), we note that (\ref{ass_der}) ensures that the vector field $H_{\im p_0}$ is globally Lipschitz on $\R^{2n}$ and therefore complete, so that the associated Hamiltonian flow exists for all times. 

Our first assumption is that the quadratic form $\la \re q\ra_{\im q,T}$ satisfies the ellipticity condition
\begin{equation}\label{ass_dyn_1}
    \la \text{Re } q \ra_{\im q,T}(X)  >0, \quad 0 \neq X\in \R^{2n},
\end{equation}
for some (and hence for all) $T>0.$

\textit{Remark.} The assumption (\ref{ass_dyn_1}) is equivalent to the condition that the quadratic form $q$ in (\ref{ass_quad}) satisfies
\begin{equation}
    \{X \in \R^{2n}: H^k_{\im q} \re q(X)=0, \quad \forall k \in \N\}=\{0\}.
\end{equation}
In the terminology of \cite{HitrikPravdaStarov2009}, \cite{hitrik_starov_i}, this is equivalent to the statement that the singular space $S \subseteq \R^{2n}$ associated to the quadratic form $q$ satisfies $S=\{0\}.$

Moreover, we suppose that there exists $T>0$ such that for every $\epsilon>0,$ there exists $\delta>0$ such that 
\begin{equation}\label{ass_dyn_2}
  |X| > \epsilon \implies \la \re p_0\ra_{\im p_0,T}(X) > \delta, \quad X \in \R^{2n}.
\end{equation} 
Here, we may notice that in view of (\ref{ass_elliptic}), the implication (\ref{ass_dyn_2}) holds automatically in a neighborhood of infinity for each $T>0.$ In what follows, we let $T>0$ be fixed and such that both (\ref{ass_dyn_1}) and (\ref{ass_dyn_2}) hold. 

Let $P=p^w(x,hD;h)$ be the semiclassical Weyl quantization of the symbol $p$ in (\ref{ass_full_symbol}), 
\begin{equation}
    p^w(x,hD;h)u(x)=\frac{1}{(2\pi h)^n} \iint_{\R^{2n}} e^{\frac{i}{h} (x-y)\cdot \theta}p\li(\frac{x+y}{2},\theta;h\ri)u(y)dyd\theta.
\end{equation}
The operator $P$ will be viewed as a closed densely defined operator on $L^2(\R^n),$ equipped with the maximal domain 
\begin{equation}\label{max_domain}
    \mathcal{D}(P)=\{u \in L^2(\R^n):Pu \in L^2(\R^n)\}=H(m),
\end{equation}
for all $h>0$ small enough. Here, $H(m)$ is the microlocally weighted Sobolev space introduced in (\ref{weighted_sobolev_defn}) and the equality $\mathcal{D}(P)=H(m)$ follows from (\ref{ass_pos}), (\ref{ass_elliptic}). It also follows from (\ref{ass_pos}), (\ref{ass_elliptic}) that for each $z \in \C$ satisfying
\begin{equation}\label{eq:spec_region}
    \re z < \frac{1}{O(1)},
\end{equation}
the operator $P-z: H(m) \to L^2(\R^n)$ is Fredholm of index zero for all $h>0$ small enough. Furthermore, using (\ref{ass_pos}), (\ref{ass_elliptic}), and the sharp G\aa rding inequality \cite[Chapter 7]{Dimassi_Sjostrand_1999}, \cite[Remark 3.5.13]{Martinez2002}, we obtain that there exists $C_0>0$ such that for all $u \in \mathcal{D}(P)$ and all $h>0$ small enough, we have 
\begin{equation}\label{eq:garding}
    \re \la Pu,u\ra_{L^2(\R^n)} \geq -C_0 h \|u\|_{L^2(\R^n)}^2.
\end{equation}
It follows from (\ref{eq:garding}) and the Fredholm property that the spectrum of $P$ is confined to the region $\{z \in \C: \re z \geq -C_0 h\},$ and we have
\begin{equation}
    \|(P-z)^{-1}\|_{\mathcal{L}(L^2(\R^n),L^2(\R^n))} \leq \frac{1}{-\re z -C_0h}, \quad \re z <-C_0h,
\end{equation}
for all $h>0$ small enough. We conclude by the analytic Fredholm theory that the spectrum of $P$ in the region (\ref{eq:spec_region}) is discrete for all $h>0$ small enough, consisting of eigenvalues of finite algebraic multiplicity.

The following is the first main result of this work. It is an analogue of \cite[Proposition 7.2]{HerauHitrikSjostrand2008}, established in that paper in the case of second order differential operators of Kramers-Fokker-Planck type.
\begin{theorem}\label{thm0}
   Assume that the assumptions 
    \textup{(\ref{ass_pos}), (\ref{ass_der}), (\ref{ass_symbol_class}), (\ref{ass_full_symbol}), 
    (\ref{ass_elliptic}), (\ref{ass_morse}), (\ref{ass_dyn_1}), (\ref{ass_dyn_2})} hold. For each $B>0,$ there exists $D>0$ and $h_0>0$ such that for all $0<h \leq h_0,$ the spectrum of $P$ does not meet the region
   \begin{equation}\label{far_set}
       \{z \in \C: \re z < Bh, |\im z-\im p_0(0)|>Dh\}.
   \end{equation}
   Furthermore, for all $z$ in the set \textup{(\ref{far_set})} and all $0<h \leq h_0,$ we have \begin{equation}\label{eq:thm1_estimate}
        \|(P-z)^{-1}\|_{\mathcal{L}(L^2(\R^n),L^2(\R^n))}\leq O\li(\frac{1}{h}\ri).
    \end{equation}
\end{theorem}

It follows from Theorem \ref{thm0} that the spectrum of $P$ in the strip $-C_0 h < \re z < Bh$ is confined to an $O(h)$--neighborhood of the critical value $p_0(0) \in i \R$ of the principal symbol of $p,$ where we expect it to be related to the spectrum of the quadratic operator $q^w(x,hD),$ with $q$ given in (\ref{ass_quad}). Before stating our next result, let us recall that it was shown in \cite{HitrikPravdaStarov2009} that under the assumption (\ref{ass_dyn_1}), the spectrum of the quadratic operator $q^w(x,D)$ on $L^2(\R^n)$ is discrete, consisting of isolated eigenvalues of finite algebraic multiplicity. The eigenvalues form a lattice in the open right half-plane, confined to an angular region of the form $|\arg z| \leq \theta <\frac\pi2$ --- see \cite[Equation (1.8)]{hitrik_starov_ii} for the explicit description of the spectrum.

The second main result of this work is the following theorem, extending \cite[Theorem 1]{hitrik_starov_i}. It establishes resolvent estimates for $P$ in an $O(h)$--neighborhood of the critical value $p_0(0) \in i \R$ of the principal symbol of $p,$ away from the eigenvalues of the quadratic approximation, shifted by the value of the subprincipal symbol at the critical point.

\begin{theorem}\label{thm1}
    Assume that the assumptions 
    \textup{(\ref{ass_pos}), (\ref{ass_der}), (\ref{ass_symbol_class}), (\ref{ass_full_symbol}), 
    (\ref{ass_elliptic}), (\ref{ass_morse}), (\ref{ass_dyn_1}), (\ref{ass_dyn_2})} hold.
    Then, for each open neighborhood $\Omega \subseteq \C$ of $\text{Spec}(q^w(x,D))$ and each $C>1,$ there exists $h_0>0$ such that for all $0 < h \leq h_0$ and $\lambda \in D(0,C)$ satisfying $\lambda \not \in \Omega,$ the resolvent $(P-p_0(0)-hp_1(0)-h\lambda)^{-1}:L^2(\R^n) \to L^2(\R^n)$ exists and satisfies
    \begin{equation}
        \|(P-p_0(0)-hp_1(0)-h\lambda)^{-1}\|_{\mathcal{L}(L^2(\R^n),L^2(\R^n))}\leq O\li(\frac{1}{h}\ri).
    \end{equation}
\end{theorem}
\textit{Example.} Consider a semiclassical non-self-adjoint Schr\"{o}dinger operator,
\begin{equation}\label{eq:schrod_op}
    P=-h^2 \Delta+V+iW, \quad V,W \in C^\infty(\R^n;\R).
\end{equation}
Then, $P=p_0^w(x,hD)$ for the complex-valued symbol $p_0(x,\xi) = |\xi|^2+V(x)+iW(x).$ Since $m(x,\xi) = 1+|\xi|^2+V(x),$ assuming that
\begin{equation}
V\geq 0, \quad \partial^\alpha V, \partial^\alpha W \in L^\infty(\R^n), \quad |\alpha| \geq 2,
\end{equation}
and 
\begin{equation}
    |W(x)| \leq C(1+V(x)), \quad X \in \R^n,
\end{equation}
for some $C>0,$ the assumptions (\ref{ass_pos}), (\ref{ass_der}), and (\ref{ass_symbol_class}) are satisfied, and $p_0 \in S(m).$
Next, the ellipticity condition (\ref{ass_elliptic}) may be translated into the assumption
\begin{equation}\label{eq:elliptic_schrod}
\liminf_{|x| \to \infty} V(x)>0.
\end{equation}   
To satisfy (\ref{ass_morse}), it suffices to assume that
\begin{equation}\label{schrod_crit_set}
    V^{-1}(0) \cap (\nabla W)^{-1}(0) = \{0\},  
\end{equation}
so that $\mathcal{C} =\{0\}.$ We then have the Taylor expansion
\begin{equation}
    p_0(X) = iW(0)+q(X)+O(|X|^3), \quad X \to 0,
\end{equation}
where
\begin{equation}\label{eq:quad_schrod}
 q(x,\xi) = |\xi|^2+\frac12 (V''(0)x \cdot x + i  W''(0)x \cdot x).   
\end{equation}
Our final assumption is that the complex symmetric $n \times n$ matrix $V''(0)+iW''(0)$ is invertible. We then claim that the dynamical conditions (\ref{ass_dyn_1}), (\ref{ass_dyn_2}) hold for $P.$ Indeed, since $\im p_0(x,\xi)=W(x)$ and $\im q(x,\xi) = \frac12 W''(0)x \cdot x,$ one has that
\begin{equation}\label{eq:hamilton_flows_pq}
    e^{tH_{\im p_0}}(x,\xi)=(x,\xi -t\nabla W(x)), \quad e^{tH_{\im q}}(x,\xi)=(x,\xi -tW''(0)x),
\end{equation}
and explicit computations using (\ref{eq:avg_defn}) and (\ref{eq:hamilton_flows_pq}) show that
\begin{equation}\label{eq:schrod_avg}
    \la \re p_0\ra_{\im p_0,T}(x,\xi) =  |\xi|^2+V(x) +\frac{T^2}{3}|\nabla W(x)|^2 
\end{equation}
and
\begin{equation}
    \la \re q\ra_{\im q,T}(x,\xi) = |\xi|^2+\frac12  V''(0)x \cdot x +\frac{T^2}{3}|W''(0)x|^2 
\end{equation}
for each fixed $T>0.$ Using the fact that $V''(0)$ is positive semi-definite and $V''(0)+iW''(0)$ is invertible, we conclude that 
\begin{equation}
    \la \re q\ra_{\im q,T}(X) \sim |X|^2,
\end{equation}
and therefore (\ref{ass_dyn_1}) holds. Moreover, (\ref{schrod_crit_set}) and (\ref{eq:schrod_avg}) imply that for each $\epsilon>0,$ the function $\la \re p_0\ra_{\im p_0,T}$ is non-vanishing in the set $|X|>\epsilon,$ so using (\ref{eq:elliptic_schrod}), we conclude that (\ref{ass_dyn_2}) also holds. We can thus apply Theorems \ref{thm0} and \ref{thm1} to the operator in (\ref{eq:schrod_op}) with $p_0(0)=iW(0),p_1(0)=0.$ Let us also recall, following \cite[Theorem 1.2.2]{HitrikPravdaStarov2009} and \cite[Theorem 3.1]{bellis_hitrik}, that the spectrum of the quadratic Schr\"{o}dinger operator $\text{Spec}(q^w(x,D)),$ with $q$ given in (\ref{eq:quad_schrod}), is given by eigenvalues of the form
\begin{equation}
    \sum_{j=1}^n \frac{\lambda_j}{i} (1+2\nu_{j,l}), \quad \nu_{j,l} \in \N,
\end{equation}
where $\lambda_j \in \C$ are such that $\im \lambda_j>0$ and $-\lambda_j^2$ are the eigenvalues of the complex symmetric invertible matrix $\frac12 (V''(0)+iW''(0))$ (see also \cite{gallagher2009spectral}, \cite{schenker2011estimating}).

Following \cite{herau_sjostrand_stolk_kfp}, \cite{hitrik_starov_i}, \cite{hitrik_starov_ii}, the principal technical tool in the proofs of Theorems \ref{thm0} and \ref{thm1} is the use of (metaplectic) FBI-Bargmann transforms and the corresponding exponentially weighted spaces of holomorphic functions on the FBI-Bargmann transform side. The exponential weights are introduced so as to take advantage of the dynamical assumptions (\ref{ass_dyn_1}), (\ref{ass_dyn_2}), increasing the real part of the principal symbol of $p$ away from $0.$
 
The plan of the paper is as follows. In Section \ref{order_function} we reduce the proofs of the main results to the case of symbols $p \in S(1)$ that are bounded with all derivatives. A construction of a globally defined bounded phase space weight function, instrumental in the proofs of Theorems \ref{thm0} and \ref{thm1}, is then carried out in Section \ref{weight_functions}. This construction, which follows closely the works \cite{herau_sjostrand_stolk_kfp}, \cite{HerauHitrikSjostrand2008}, \cite{hitrik_starov_i}, is due to \cite{Stone2022} and is given here with the kind permission of the author. Basic facts concerning metaplectic FBI-Bargmann transforms, I-Lagrangian deformations of the real phase space $\R^{2n},$ and the associated exponentially weighted spaces of holomorphic functions are reviewed briefly in Section \ref{fbi_and_manifolds}. Section \ref{ext_estimate} is devoted to the proof of Theorem \ref{thm0}. A priori estimates for elliptic quadratic operators acting on quadratic Bargmann spaces are discussed in Section \ref{local_estimate}, as a preparation for the proof of Theorem \ref{thm1}, and this theorem is then established in Section \ref{glue_together}. Appendix \ref{appendix_a} discusses microlocally  weighted Sobolev spaces on $\R^n$ and their characterization on the FBI-Bargmann transform side.

\textbf{Acknowledgements.} I am most grateful to Matthew Stone for allowing me to include in Section \ref{weight_functions} the proof of Proposition \ref{G_prop}, taken from the unpublished Ph.D. thesis \cite{Stone2022}.
 
\section{Reduction to bounded symbols}\label{order_function}
The existence of a pseudodifferential calculus requires symbols with a sufficient amount of regularity and control of derivatives (see \cite[Chaper 18]{Hormander1985}). In the semiclassical framework, one typically deals with symbol classes (see, for example, Chapter 18 of \cite{Hormander1985})
\begin{equation}
 S(m) =\{a \in C^\infty(\R^{2n}):\forall \alpha \in \N^{2n}, \exists C_\alpha>0, |\partial^\alpha a|\leq C_\alpha m\}  
\end{equation}
for
a function $m: \R^{2n} \to (0,\infty)$ satisfying
\begin{equation}
    \exists C>0, \exists N \in \R, m(Y) \leq C \la X-Y \ra^N m(X), \quad X,Y \in \R^{2n},
\end{equation}
which is henceforth known as an \textit{order function} (see \cite[Chapter 7]{Dimassi_Sjostrand_1999}, \cite[Chapter 4]{Zworski2012}). Here, 
\begin{equation}
    \la X\ra:=(1+|X|^2)^{\frac12}
\end{equation}
is the Japanese bracket of $X.$ Our starting point is the following result:
\begin{lemma}\label{order_function_lemma}
    Suppose that $p_0 \in C^\infty(\R^{2n})$ satisfies the assumptions \textup{(\ref{ass_pos}), (\ref{ass_der}), (\ref{ass_symbol_class}).} Then, $m=1+\re p_0$ is an order function, and $p_0 \in S(m).$
\end{lemma}
\begin{proof}
    We must first show that $m=1+\re p_0$ is an order function. Indeed, by Taylor's theorem, (\ref{ass_pos}) and (\ref{ass_der}), and the classic result that 
    \begin{equation}\label{eq:double_derivative_estimate}
    0 \leq f \in C^2(\R^n), \quad |\nabla^2f| \in L^\infty(\R^n) \implies |\nabla f| =O(\sqrt{f}), 
    \end{equation}
    (see \cite[Lemma 4.31]{Zworski2012}), 
    we deduce that
    \begin{equation}
        \begin{split}
            m(X) & \leq m(Y)+ \nabla m(Y) \cdot (X-Y) + O(|X-Y|^2) \\
            & \leq m(Y) + C\sqrt{m(Y)}|X-Y| +O(|X-Y|^2) \\
            & \leq m(Y)+Cm(Y)+O(|X-Y|^2)
            \leq Cm(Y) \la X-Y\ra^2
        \end{split}
    \end{equation}
    using the fact that $m \geq 1.$
    This shows that $m$ is indeed an order function, so it remains to show that $p_0 \in S(m).$ Indeed, the symbolic estimates in the case $|\alpha|=0$ follow from (\ref{ass_symbol_class}),   while the case $|\alpha| \geq 2$ follows from (\ref{ass_der}) and the fact that $m \geq 1.$ The remaining case $|\alpha|=1$ then follows from (\ref{ass_pos}), (\ref{ass_der}), and (\ref{eq:double_derivative_estimate}), as
    \begin{equation}
        |\nabla \re p_0| \leq O(\sqrt{\re p_0}) \leq O(\sqrt{m}) \leq O(m),
    \end{equation}
    while (\ref{ass_symbol_class}) also yields 
    \begin{equation}
        |\nabla \im p_0| \leq |\nabla(Cm-\im p_0)|+O(|\nabla m|)\leq \sqrt{O(m)-\im p_0}+O(\sqrt{m}) \leq O(\sqrt{m}) \leq O(m).
    \end{equation}
    This completes the proof.
\end{proof}
We now demonstrate that it suffices to prove Theorems \ref{thm0} and \ref{thm1} in the case of symbols $p \in S(1).$ Indeed, assume that Theorems \ref{thm0}, \ref{thm1} have already been established in the case when $p_0 \in S(1)$ and the  assumptions (\ref{ass_pos}), (\ref{ass_full_symbol}), (\ref{ass_elliptic}), (\ref{ass_morse}), (\ref{ass_dyn_1}), (\ref{ass_dyn_2}) are satisfied with $m=1.$ Given $p \in S(m),$ with $m=1+\re p_0,$ satisfying the assumptions of Theorems \ref{thm0}, \ref{thm1}, let us set
\begin{equation}
    \overline{p}:= \chi p+(1-\chi)
\end{equation}
for a cutoff $\chi \in C_c^\infty(\R^{2n};[0,1])$ such that $\chi(X)=1$ in a neighborhood of $\{X:|X| \leq R\},$ where $R>0$ is large enough so that $(\re p_0)^{-1}(0) \Subset \{X:|X|<R\}.$ Then, $\overline{p} \in S(1),$ and (\ref{ass_full_symbol}) yields the expansion 
\begin{equation}
    \overline{p}(x,\xi) \sim \sum_{k=0}^\infty h^k \overline{p}_k(x,\xi), \quad \overline{p}_k := \chi p_k+(1-\chi),
\end{equation}
in the space $S(1).$ We now verify that the assumptions of Theorems \ref{thm0}, \ref{thm1} hold for $\overline{p}.$ The leading symbol $\overline{p}_0 = \chi p_0+(1-\chi) \in S(1)$ satisfies $\re \overline{p}_0 \geq 0$ and 
\begin{equation}
    \exists C>0, \quad |X|\geq C \implies \re \overline{p}_0(X) \geq \frac1C.
\end{equation}
Moreover,
\begin{equation}
    \{X \in \R^{2n}:\re \overline{p}_0(X) = H_{\im \overline{p}_0}(X)=0\}=\{X \in \R^{2n}: \re p_0(X)= H_{\im p_0}(X)=0\} = \{0\},
\end{equation}
so (\ref{ass_morse}) also holds for $\overline{p}_0.$ 

We have $\overline{p}_0=p_0$ in a neighborhood of the origin, and therefore
\begin{equation}
    \overline{p}_0(X)= p_0(0)+q(X)+O(|X|^3), \quad X \to 0,
\end{equation}
so the quadratic form $q$ satisfies (\ref{ass_dyn_1}). It thus remains to verify the dynamical assumption (\ref{ass_dyn_2}) for $\overline{p}_0,$ provided that it holds for $p_0.$
\begin{prop}
    There exists some $T>0$ such that for every $\epsilon>0,$ there exists $\delta >0$ such that
    \begin{equation}
  |X| > \epsilon \implies \la \re \overline{p}_0\ra_{\im \overline{p}_0,T}(X) > \delta.  
\end{equation}
\end{prop}
\begin{proof}
Let $T>0$ be fixed and such that (\ref{ass_dyn_2}) holds for $p_0.$ Arguing as in \cite[Proposition 2]{hitrik_starov_i} and using that $H_{\im \overline{p}_0}(0)=0,$ we conclude that there exists $C_T>0$ such that \begin{equation}\label{eq:gronwall}
    |e^{t H_{\im \overline{p}_0}}(X)-X| \leq C_T t|X|, \quad |t| \leq T, \quad X \in \R^{2n}.
\end{equation} 
We set
\begin{equation}
    R' = \frac{R}{C_T T+1}
\end{equation}
and note that (\ref{eq:gronwall}) implies that
\begin{equation}\label{eq:flow_bound}
 |X| \leq R' \implies |e^{t H_{\im \overline{p}_0}}(X)| \leq R, \quad |t| \leq T.     
\end{equation}
By increasing $R$ if necessary, we may pick $R'$ so that $\re \overline{p}_0$ is elliptic in the region $|X| \geq \frac{R'}{2}.$ 

Next, for $|X| \leq R'$ and $|t| \leq T,$ let us write 
\begin{equation}
    e^{tH_{\im p_0}}X = X + \int_0^t H_{\im p_0}(e^{sH_{\im p_0}}X)ds
\end{equation}
and 
\begin{equation}
    e^{tH_{\im \overline{p}_0}}X = X + \int_0^t H_{\im \overline{p}_0}(e^{sH_{\im \overline{p}_0}}X)ds.
\end{equation}

Then,
\begin{equation}\label{eq:ode_split}
\begin{split}
     e^{tH_{\im p_0}}X -  e^{tH_{\im \overline{p}_0}}X & = \int_0^t (H_{\im p_0} (e^{sH_{\im p_0} }X)-H_{\im \overline{p}_0}(e^{sH_{\im \overline{p}_0}}X))ds \\ &
     = \int_0^t (H_{\im p_0}(e^{sH_{\im p_0}}X)-H_{\im p_0}(e^{sH_{\im \overline{p}_0}}X))ds\\
     & +\int_0^t (H_{\im p_0}(e^{sH_{\im \overline{p}_0}}X)-H_{\im \overline{p}_0}(e^{sH_{\im \overline{p}_0}}X))ds.
\end{split}
\end{equation}
Now, since $p_0 =\overline{p}_0$ in a neighborhood of $\{X:|X| \leq R\},$ one has that $H_{\im p_0}=H_{\im \overline{p}_0}$ in a neighborhood of $\{X:|X| \leq R\},$ and hence for $|X| \leq R'$ and $|t| \leq T,$ (\ref{eq:flow_bound}) implies that the second term on the right hand side of (\ref{eq:ode_split}) vanishes. Then, using that $H_{\im p_0}$ is globally Lipschitz yields 
\begin{equation}
    |e^{tH_{\im p_0}}X -  e^{tH_{\im \overline{p}_0}}X| \leq C\int_0^t|e^{sH_{\im p_0}}X -  e^{sH_{\im \overline{p}_0}}X| ds, \quad 0 \leq t \leq T, \quad |X| \leq R',
\end{equation}
with a similar inequality for $-T \leq t \leq 0.$ A standard Gr\"{o}nwall-type argument then implies that 
\begin{equation}
    e^{t H_{\im p_0}}X = e^{t H_{\im \overline{p}_0}}X, \quad |X| \leq R', \quad |t| \leq T,
\end{equation}
and hence
\begin{equation}\label{eq:avg_near}
    \la \re \overline{p}_0\ra_{\im \overline{p}_0,T}(X)=\la \re p_0\ra_{\im p_0,T}(X), \quad |X| \leq R'.
\end{equation}

On the other hand,  (\ref{eq:gronwall}) implies that
\begin{equation}
    |e^{t H_{\im \overline{p}_0}}X| \geq \frac{|X|}{2}, \quad |t| \leq \min\li(\frac{1}{2C_T}, T\ri),
\end{equation}
and hence
\begin{equation}
    |e^{t H_{\im \overline{p}_0}}X| \geq \frac{R'}{2}, \quad |t| \leq \min\li(\frac{1}{2C_T}, T\ri),
\end{equation}
for $|X| \geq R'.$ Since $\re \overline{p}_0$ is elliptic in the region $|X| \geq \frac{R'}{2},$ we thus obtain that 
\begin{equation}\label{eq:avg_far}
    \la \re \overline{p}_0\ra_{\im \overline{p}_0,T}(X) \geq \frac{1}{C}, \quad |X| \geq R',
\end{equation}
for some $C>0.$ Then, combining (\ref{eq:avg_near}) and (\ref{eq:avg_far}) concludes the proof of the proposition.
\end{proof}
Following \cite[Section 1]{hitrik_starov_i}, we now explain how one may obtain the resolvent estimate for $P=p^w(x,hD;h)$ in Theorems \ref{thm0}, \ref{thm1} from the corresponding resolvent estimate for $\overline{P}:=\overline{p}(x,hD;h).$ When doing so, we may assume that $p_0(0)=0,$ and will discuss the corresponding reduction in the case of Theorem \ref{thm1}. Fix any open neighborhood $\Omega \subseteq \C$ of $\text{Spec}(q^w(x,D))$ and $C>1.$ Then, the resolvent estimate for $\overline{P}$ yields for all $h>0$ small enough, all $\lambda \in D(0,C)$ satisfying $\lambda -p_1(0) \not \in \Omega,$ and all $u \in \s(\R^n),$ we have the a priori estimate 
\begin{equation}\label{eq:apriori}
 h\|u\| \leq O(1)\|(\overline{P}-h\lambda)u\|.
\end{equation}
Here and below, $L^2(\R^n)$ norms are used throughout. Let $\widetilde{\chi} \in C_c^\infty(\R^{2n};[0,1])$ be a smooth cutoff such that $\widetilde{\chi} = 1$ in a neighborhood of $(\re p_0)^{-1}(0)$ and $\chi =1$ near $\text{supp } \widetilde{\chi}.$ Using the ellipticity of $p-h\lambda$ near the support of $1-\widetilde{\chi},$ we may construct a pseudodifferential parametrix $A,$ bounded on $L^2(\R^n)$ uniformly in $h,$ such that
\begin{equation}
    A(P-h\lambda) = 1-\widetilde{\chi}^w +O_{L^2 \to L^2}(h^\infty).
\end{equation}
Thus, we obtain
\begin{equation}\label{eq:micro_inv_1}
    \|(1-\widetilde{\chi}^w)u\| \leq O(1)\|(P-h\lambda)u\|+O(h^\infty)\|u\|
\end{equation}
for all $u \in \s(\R^n).$
Moreover, (\ref{eq:apriori}) yields that
\begin{equation}\label{eq:micro_inv_2}
    h \|\widetilde{\chi}^wu\|\leq O(1)\|(\overline{P}-h\lambda)\widetilde{\chi}^wu\|\leq  O(1)\|(P-h\lambda)\widetilde{\chi}^w u\| +O(h^\infty)\|u\|,
\end{equation}
where we use the fact that $(P-\overline{P}) \widetilde{\chi}^w = O(h^\infty)$ in $\mathcal{L}(L^2(\R^n),L^2(\R^n)),$ since the supports of $p-\overline{p}$ and $\widetilde{\chi}$ are disjoint. 
On the other hand,
\begin{equation}\label{eq:micro_inv_3}
    \|(P-h\lambda)\widetilde{\chi}^w u\| \leq \|\widetilde{\chi}^w(P-h\lambda)u\| +\|[P,\widetilde{\chi}^w]u\| \leq O(1)\|(P-h\lambda)u\| +\|[P,\widetilde{\chi}^w]u\|.
\end{equation}
We note that the commutator $[P,\widetilde{\chi}^w]$ is compactly microlocalized in the region $\text{supp}(\nabla \widetilde{\chi}),$ which follows from the asymptotic expansion for the symbol of the commutator and the fact that it can be expressed in terms of the derivatives of $\widetilde{\chi}$. Letting $\psi \in C_c^\infty(\R^{2n};[0,1])$ be such that $\psi =1$ in a neighborhood of $(\re p_0)^{-1}(0)$ and $\widetilde{\chi}=1$ near $\text{supp }\psi,$ we thus get
\begin{equation}\label{eq:micro_inv_4}
\begin{split}
     \|[P,\widetilde{\chi}^w]u\| & \leq\|[P,\widetilde{\chi}^w](1-\psi^w)u\| + \|[P,\widetilde{\chi}^w]\psi^wu\| \\ & \leq O(1) \|(1-\psi^w)u\|+O(h^\infty)\|u\| \leq O(1)\|(P-h\lambda)u\| +O(h^\infty)\|u\|,
\end{split}
\end{equation}
where we use the uniform boundedness of the commutator on $L^2(\R^n)$ and the ellipticity of $p-h\lambda$ near the support of $1-\psi.$

Putting the estimates (\ref{eq:micro_inv_1}), (\ref{eq:micro_inv_2}), (\ref{eq:micro_inv_3}), and (\ref{eq:micro_inv_4}) together, we thus obtain 
\begin{equation}
    h\|u\| \leq h\|\widetilde{\chi}^w u\| +h\|(1-\widetilde{\chi}^w)u\| \leq O(1)\|(P-h\lambda)u\| +O(h^\infty)\|u\|. 
\end{equation}
Hence, for small enough $h>0,$ one deduces the a priori estimate  \begin{equation}\label{eq:bdd_below}
    h\|u\| \leq O(1) \|(P-h\lambda)u\|
\end{equation}
for all $u \in \s(\R^n).$

By density, the estimate (\ref{eq:bdd_below}) extends to all $u \in \mathcal{D}(P)=H(m).$ It follows that $P-h\lambda: H(m)\to L^2(\R^n)$ is injective and has closed range, and thus invertible in view of the Fredholm property. We thus conclude that $P-h \lambda :H(m) \to L^2(\R^n)$ is bijective, and hence the resolvent $(P-h \lambda)^{-1} :L^2(\R^n) \to L^2(\R^n)$ is well-defined and satisfies the bound
\begin{equation}
    \|(P-h\lambda)^{-1}\|_{\mathcal{L}(L^2(\R^n),L^2(\R^n))} \leq \frac{C}{h}
\end{equation}
for some $C>0.$ We shall henceforth prove both Theorems \ref{thm0} and \ref{thm1} for symbols $p \in S(1).$

\section{Method of averaging and bounded exponential weights}\label{weight_functions}
The purpose of this section is to construct a globally defined phase space exponential weight $G=G_\epsilon \in C_c^\infty(\R^{2n};
\R),$ which will play a crucial role in the proof of Theorems \ref{thm0}, \ref{thm1}. Here, $0<\epsilon \ll 1$ is a small parameter which will be later chosen to be proportional to $h.$ The following construction is inspired by closely related techniques in the previous works of \cite{herau_sjostrand_stolk_kfp}, \cite{HerauHitrikSjostrand2008}, \cite{hitrik_starov_i}, following a long tradition of works on resonances and non-self-adjoint operators, such as \cite{helffer_sjostrand_resonances} and \cite{MelinSjostrand2003}. In particular, the exponential weights will be introduced to compensate for the lack of ellipticity of the principal symbol in (\ref{ass_pos}), (\ref{ass_der}), (\ref{ass_symbol_class}), both at the level of the quadratic approximation at the doubly characteristic point $(0,0) \in \R^{2n}$ as well as in a compact set away from the origin. The results of this section have been established in the unpublished Ph.D. thesis \cite{Stone2022}, and are reproduced here with the kind permission of the author.  

Let $p_0 \in S(1),$ and consider an almost holomorphic extension $\widetilde{p}_0 \in C^\infty(\C^{2n})$ of $p_0$ supported in a tubular neighborhood of $\R^{2n}$ in $\C^{2n}$ and such that $\partial^\alpha \widetilde{p}_0 \in L^\infty(\C^{2n})$ for all $\alpha \in \N^{4n}.$ Thus, $\widetilde{p}_0|_{\R^{2n}}=p_0,$ and one has that 
\begin{equation}
    |\overline{\partial}\widetilde{p}_0(x,\xi)| \leq C_N |\im (x,\xi)|^N
\end{equation}
for all $N \in \N$ (see \cite[Section 1]{MelinSjoestrand1975}, \cite[Section 1.4]{almost_holomorphic_ref}).

The following proposition, due to \cite{Stone2022}, is the main result of this section.
\begin{prop}\label{G_prop}
    Let $p_0 \in S(1)$ be such that the assumptions \textup{(\ref{ass_pos}), (\ref{ass_elliptic}), (\ref{ass_morse}), (\ref{ass_dyn_1}),} and \textup{(\ref{ass_dyn_2})} hold, and let $\widetilde{p}_0 \in C_b^\infty(\C^{2n})$ be an almost holomorphic extension of $p_0,$ as above. 
    Then, there exist constants $C > 1, 0<\delta_0 \leq 1, 0<\epsilon_0 \leq 1,$ and a function $G_\epsilon \in C_c^\infty(\R^{2n};\R),$ depending on the parameter $0<\epsilon \leq \epsilon_0,$ satisfying
    \begin{equation}\label{eq:g_eps_control}
        |\partial^\alpha G_\epsilon(X)|=O(\epsilon^{1-\frac{|\alpha|}{2}}), \quad |\alpha| \leq 2,
    \end{equation}
    uniformly on $\R^{2n}.$ Moreover, for all $0<\epsilon \leq \epsilon_0, 0 < \delta \leq \delta_0,$ one has  \begin{equation}\label{eq:elliptic_local_final}
        \re (\widetilde{p}_0(X+i \delta H_{G_\epsilon}(X))) \geq \frac{\delta |X|^2}{C}, \quad X \in \R^{2n},\quad |X| \leq \epsilon^{\frac12},
    \end{equation}
    \begin{equation}\label{eq:elliptic_exterior_final}
        \re (\widetilde{p}_0(X+i \delta H_{G_\epsilon}(X))) \geq \frac{\delta \epsilon}{C}, \quad X \in \R^{2n}, \quad |X| \geq \epsilon^{\frac12}.
    \end{equation}
\end{prop}
While proving Proposition \ref{G_prop}, we may assume that $p_0(0)=0.$ Our starting point is the following result, established in \cite[Proposition 2]{hitrik_starov_i}: 

\begin{lemma}\label{lemma_avg}
    For each fixed $T>0,$ one has
    \begin{equation}
        \la \re p_0\ra_{\im p_0,T}(X) = \la \re q\ra_{\im q,T}(X)+O_T(|X|^3), \quad |X| \to 0.
    \end{equation}
\end{lemma}

We now introduce a modified symbol associated to $\re p_0,$ which will be used in the construction of the weight function $G_\epsilon$. Let $g \in C^\infty([0,\infty); [0,1])$ be a decreasing function such that 
\begin{equation}
    g(t) = 1, \quad t \in [0,1]
, \quad g(t) = t^{-1}, \quad t \geq 2.\end{equation}
Then, for each $k \in \N,$ one has 
\begin{equation}\label{eq:g_der_bound}
    |g^{(k)}(t)| = O\li(\frac{1}{\la t\ra^{k+1}}\ri)
\end{equation}
as $t \to \infty.$ Now, letting $\chi_0 \in C_c^\infty(\R^{2n};[0,1])$ be a smooth cutoff such that $\chi_0 = 1$ in a small neighborhood near the origin, we define the modified symbol
\begin{equation}\label{eq:modified_symbol_defn}
    (\re p_0)_\epsilon(X) := \chi_0(X)g\li(\frac{|X|^2}{\epsilon}\ri) \re p_0(X) + \epsilon(1-\chi_0(X))\re p_0(X).
\end{equation}
It follows from (\ref{eq:modified_symbol_defn}) that the constructed symbol satisfies
\begin{equation}\label{eq:interp_bound}
    0\leq (\re p_0)_\epsilon \leq \re p_0,
\end{equation}
and we have 
\begin{equation}\label{eq:symbol_estimates}
\begin{split}
    (\re p_0)_\epsilon(X) & = \re p_0(X), \quad |X| \leq \epsilon^{\frac12}, \\
    (\re p_0)_\epsilon(X) & = g\li(\frac{|X|^2}{\epsilon}\ri)\re p_0(X) \sim \frac{\epsilon}{|X|^2} \re p_0(X), \quad \epsilon^{\frac12} \leq |X| \leq \frac{1}{C}, \\
    (\re p_0)_\epsilon(X) & = \chi_0(X) \frac{\epsilon}{|X|^2} \re p_0(X)+\epsilon (1-\chi_0(X)) \re p_0(X)\sim \epsilon \re p_0(X), \quad |X|  \geq \frac{1}{C},
\end{split}
\end{equation}
for some $C>0$ large enough depending on  $\chi_0,$ but not on $\epsilon.$

Direct computations using (\ref{ass_quad}), (\ref{eq:g_der_bound}), and (\ref{eq:symbol_estimates}) imply that 
\begin{equation}\label{eq:eps_bound_on_symbol}
    |\partial^\alpha (\re p_0)_\epsilon(X)| = O(\epsilon^{1-\frac{|\alpha|}{2}}), \quad |\alpha| \leq 2,
\end{equation}
uniformly on $\R^{2n}.$
    
We shall now introduce a bounded weight function. For $T>0$ fixed, we set 
\begin{equation}\label{eq:weight_fun_defn}
    G_\epsilon(X) := \int_{-\infty}^\infty J\li(\frac{t}{T}\ri) (\re p_0)_\epsilon (e^{t H_{\im p_0}}(X))dt,
\end{equation}
where $J: \R \to \R$ is a piecewise linear compactly supported function satisfying
\begin{equation}\label{eq:j_prime}
    J'(t) = \delta(t)-\frac12 1_{[-1,1]}(t).
\end{equation}
Here, $1_{[-1,1]}$ is the characteristic function of the interval $[-1,1].$ The utility of the weight function is demonstrated in the following lemma:
\begin{lemma}
    The weight function $G_\epsilon$ satisfies
    \begin{equation}\label{eq:lemma_mod_symbol}
        H_{\im p_0} G_\epsilon = \la (\re p_0)_\epsilon\ra_{\im p_0,T} - (\re p_0)_\epsilon.
    \end{equation}
\end{lemma}
\begin{proof}
By the definition of the Hamilton flow, one may apply integration by parts and (\ref{eq:j_prime}) to obtain
\begin{equation}
\begin{split}
     H_{\im p_0}G_\epsilon(X)
     & = \partial_s|_{s=0} G_\epsilon(e^{s H_{\im p_0}}(X))\\ & = 
     \int_{-\infty}^\infty J\li(\frac{t}{T}\ri)\partial_s|_{s=0}\li[(\re p_0)_\epsilon(e^{(t+s)H_{\im p_0}}(X))\ri]dt \\
     & = \int_{-\infty}^\infty J\li(\frac{t}{T}\ri)\partial_t\li[(\re p_0)_\epsilon(e^{tH_{\im p_0}}(X))\ri]dt \\
     & = -\int_{-\infty}^\infty \partial_t \li[J\li(\frac{t}{T}\ri)\ri](\re p_0)_\epsilon(e^{t H_{\im p_0}}(X))dt \\
     & = -\frac{1}{T}\int_{-\infty}^\infty \li[T\delta(t)-\frac{1_{[-T,T]}(t)}{2}\ri](\re p_0)_\epsilon (e^{tH_{\im p_0}}(X)) dt \\
     & = \la (\re p_0)_\epsilon\ra_{\im p_0,T}(X) - (\re p_0)_\epsilon(X).
\end{split}
\end{equation}
\end{proof}
Then, (\ref{eq:eps_bound_on_symbol}) and differentiating (\ref{eq:weight_fun_defn}) under the integral sign yields
\begin{equation}\label{eq:G_der_bound}
    |\partial^\alpha G_\epsilon(X)| =O(\epsilon^{1-\frac{|\alpha|}{2}}), \quad |\alpha| \leq 2,
\end{equation}
uniformly for $X\in \R^{2n}$. Here we also use that
\begin{equation}
    |\partial^\alpha_X (e^{tH_{\im p_0}}X)| \leq O_T(1), \quad 1 \leq |\alpha| \leq 2, \quad |t| \leq T, \quad X \in \R^{2n}.
\end{equation}
We note that since
\begin{equation}
    |e^{tH_{\im p_0}}(X)| \leq e^{K|t|}|X|, \quad X \in \R^{2n}, \quad t \in \R,  
\end{equation}
for some $K>0,$ the first estimate in (\ref{eq:symbol_estimates}) and (\ref{eq:lemma_mod_symbol}) together imply that 
\begin{equation}
    H_{\im p_0}G_\epsilon(X) = \la \re p_0\ra_{\im p_0,T}(X)-\re p_0(X), \quad |X| \leq \frac{\epsilon^{\frac12}}{C},
\end{equation}
for some $C>1$ depending on $K$ and $T$ only. Furthermore, using (\ref{eq:weight_fun_defn}), we get that
\begin{equation}
    G_\epsilon(X) := \int_{-\infty}^\infty J\li(\frac{t}{T}\ri) \re p_0 (e^{t H_{\im p_0}}(X))dt, \quad |X| \leq \frac{\epsilon^{\frac12}}{C},
\end{equation}
so that using Lemma \ref{lemma_avg} as in \cite[Section 2]{hitrik_starov_i}, one gets 
\begin{equation}\label{eq:G_eps_G_zero_Taylor}
    G_\epsilon(X)=G_0(X)+O(|X|^3), \quad |X| \leq \frac{\epsilon^{\frac12}}{C}, 
\end{equation}
where
\begin{equation}\label{eq:G_zero_defn}
    G_0(X):= \int_{-\infty}^\infty J\li(\frac{t}{T}\ri) \re q(e^{tH_{\im q}}(X))dt.
\end{equation}
In particular, we have
\begin{equation}\label{eq:g_eps_der_control}
   G_\epsilon(X)=O(|X|^2), \quad \nabla G_\epsilon(X) = O(|X|), \quad |X| \leq \frac{\epsilon^{\frac12}}{C}.
\end{equation}

It remains to demonstrate the improved ellipticity estimates (\ref{eq:elliptic_local_final}) and (\ref{eq:elliptic_exterior_final}). Let $\widetilde{p}_0$ be an almost holomorphic extension of the principal symbol $p_0,$ as described above. Then, for $\delta>0$ sufficiently small, Taylor's theorem implies that
\begin{equation}\label{eq:pre_taylor}
\begin{split}
    \widetilde{p}_0(X+i\delta H_{G_{\epsilon}}(X)) & = p_0(X) + i\delta(\partial \widetilde{p}_0(X) \cdot H_{G_\epsilon}(X))+O(\delta^2|H_{G_\epsilon}(X)|^2) \\
    & = p_0(X) + i\delta H_{G_\epsilon}p_0(X)+O(\delta^2|\nabla G_\epsilon(X)|^2), \quad X \in \R^{2n}.
\end{split} 
\end{equation}
Taking the real part and using (\ref{eq:lemma_mod_symbol}) gives 
\begin{equation}\label{eq:hol_taylor}
\begin{split}
  \re \widetilde{p}_0(X+i\delta H_{G_\epsilon}(X))&=\re p_0(X)+\delta H_{\im p_0}G_\epsilon(X) + O(\delta^2 |\nabla G_\epsilon(X)|^2) \\
    & = \re p_0(X)+\delta (\la (\re p_0)_\epsilon\ra_{\im p_0,T}(X) - (\re p_0)_\epsilon(X)) + O(\delta^2 |\nabla G_\epsilon(X)|^2).
\end{split}
\end{equation}
For a sufficiently large constant $C>1,$ we consider the expression on the right hand side of (\ref{eq:hol_taylor}) in four different regions, as follows: $|X|^2 \leq \frac{\epsilon}{C},$ the intermediate region $\frac{\epsilon}{C}\leq |X|^2 \leq C \epsilon,$ the region $C\epsilon \leq |X|^2 \leq \frac{1}{C},$ and the unbounded region $|X|^2 \geq \frac{1}{C}.$

We first handle the region $|X|^2 \leq \frac{\epsilon}{C},$ where 
\begin{equation}
    (\re p_0)_\epsilon = \re p_0, \quad \la(\re p_0)_\epsilon\ra_{\im p_0,T} = \la\re p_0\ra_{\im p_0,T}.
\end{equation}
In this region, Lemma \ref{lemma_avg}, (\ref{eq:g_eps_der_control}), and (\ref{eq:hol_taylor}) then imply
\begin{equation}
\begin{split}
     \re \widetilde{p}_0(X+i\delta H_{G_\epsilon}(X)) & =\re p_0(X)+\delta (\la \re p_0\ra_{\im p_0,T}(X) - \re p_0(X)) + O(\delta^2 |\nabla G_\epsilon(X)|^2) \\
     & = (1-\delta) \re p_0(X)+\delta \la \re q \ra_{\im q,T}(X)+O(\delta|X|^3+\delta^2|X|^2).
\end{split}
\end{equation}
Since $\re p_0 \geq 0$ and $\la \re q\ra_{\im q,T}(X) \sim |X|^2,$ we obtain in view of (\ref{ass_dyn_1}) that there exists $\widetilde{C}>0$ such that
\begin{equation}\label{eq:elliptic_local}
    \re \widetilde{p}_0(X+i\delta H_{G_\epsilon}(X)) \geq \frac{\delta|X|^2}{\widetilde{C}}-O(\delta |X|^3+\delta^2 |X|^2) \geq \frac{\delta|X|^2}{2\widetilde{C}}, \quad |X|^2 \leq \frac{\epsilon}{C},
\end{equation}
provided that $\delta>0$ and $\epsilon>0$ are small enough.

Next, we consider the region $C\epsilon \leq |X|^2 \leq \frac{1}{C}.$ Here we have
\begin{equation}
    (\re p_0)_\epsilon(X) = g\li(\frac{|X|^2}{\epsilon}\ri)\re p_0(X)=\frac{\epsilon}{|X|^2} \re p_0(X)
\end{equation}
and 
\begin{equation}
    |e^{tH_{\im p_0}}X|^2 \geq 2\epsilon, \quad \chi_0(e^{tH_{\im p_0}}X)=1, \quad |t| \leq T,
\end{equation}
provided that $C>1$ is large enough. Using (\ref{eq:G_der_bound}) and
(\ref{eq:hol_taylor}), in this region we get
\begin{equation}\label{eq:elliptic_exterior}
\begin{split}
  \re \widetilde{p}_0(X+i\delta H_{G_\epsilon}(X))&=\re p_0(X)-\delta(\re p_0)_\epsilon(X)+\frac{\delta \epsilon}{2T} \int_{-T}^T \frac{\re p_0(e^{t H_{\im p_0}}(X))}{|e^{tH_{\im p_0}}(X)|^2}dt + O(\delta^2 \epsilon).
\end{split}
\end{equation}
To better understand the integral on the right hand side of (\ref{eq:elliptic_exterior}), we recall from \cite[Equation (2.12)]{hitrik_starov_i} that there exists $c>0$ such that
\begin{equation}
|e^{t H_{\im p_0}}(X)-e^{t H_{\im q}}(X)|\leq ct|X|^2 , \quad |t| \leq T, \quad X \in \R^{2n}.    
\end{equation} 
It follows that
\begin{equation}\label{eq:int_avg_1}
    \frac{1}{2T} \int_{-T}^T \frac{\re p_0(e^{t H_{\im p_0}}(X))}{|e^{tH_{\im p_0}}(X)|^2}dt =\frac{1}{2T} \int_{-T}^T \frac{\re q(e^{t H_{\im q}}(X))}{|e^{tH_{\im p_0}}(X)|^2}dt +O(|X|)
\end{equation}
and
 \begin{equation}\label{eq:int_avg_2}
     \frac{1}{|e^{tH_{\im p_0}}(X)|^2} - \frac{1}{|e^{tH_{\im q}}(X)|^2} = O\li(\frac{t|X|^3}{|X|^4}\ri), \quad |t| \leq T.
 \end{equation}
 Therefore, combining (\ref{eq:int_avg_1}) and (\ref{eq:int_avg_2}) gives
 \begin{equation}\label{eq:avg_f_defn}
     \frac{1}{2T} \int_{-T}^T \frac{\re p_0(e^{t H_{\im p_0}}(X))}{|e^{tH_{\im p_0}}(X)|^2}dt =f(X)+O(|X|),
 \end{equation}
 where
 \begin{equation}
     f(X)= \frac{1}{2T} \int_{-T}^T \frac{\re q(e^{t H_{\im q}}(X))}{|e^{tH_{\im q}}(X)|^2}dt, \quad X \not =0.
 \end{equation}
We note that $f$ is non-negative and homogeneous of degree zero, and it follows from (\ref{ass_dyn_1}) that $f$ is in fact strictly positive on $\R^{2n} \setminus \{0\}.$ Thus, for some $D>0$ we have
\begin{equation}\label{eq:homog_fun}
    f(X) \geq \frac{1}{D}, \quad X \not =0.
\end{equation}
Since $0 \leq g(t) \leq 1$ for all $t,$ one has $\re p_0(X) - \delta(\re p_0)_\epsilon(X) \geq 0$ for all $0<\delta \leq 1,$ so by combining (\ref{eq:elliptic_exterior}) with (\ref{eq:avg_f_defn}) and (\ref{eq:homog_fun}), we get 
\begin{equation}
    \re \widetilde{p}_0(X+i\delta H_{G_\epsilon}(X)) \geq \frac{\delta \epsilon}{D} -O(\delta^2 \epsilon)-O(\delta \epsilon |X|)
\end{equation}
in the region $C \epsilon \leq |X|^2 \leq \frac{1}{C}.$ As a result,
 for $\delta>0$ sufficiently small and $C>1$ sufficiently large,
\begin{equation}
     \re \widetilde{p}_0(X+i\delta H_{G_\epsilon}(X)) \geq \frac{\delta \epsilon}{\widetilde{C}}, \quad C\epsilon \leq |X|^2 \leq \frac{1}{C}
\end{equation}
for some $\widetilde{C}>0.$

We will next handle the intermediate region $\frac{\epsilon}{C} \leq |X|^2 \leq C\epsilon.$  In this region, we have 
\begin{equation}
 g\li(\frac{|e^{tH_{\im p_0}}(X)|^2}{\epsilon}\ri) \sim 1   
\end{equation}
uniformly in $|t| \leq T$ and $0 < \epsilon \leq 1.$ It follows from (\ref{eq:hol_taylor}) that for $\delta>0$ and $\epsilon>0$ small enough,
\begin{equation}\label{eq:elliptic_intermediate}
\begin{split}
  \re \widetilde{p}_0(X+i\delta H_{G_\epsilon}(X))&=\re p_0(X)-\delta(\re p_0)_\epsilon(X) \\ & +\frac{\delta }{2T} \int_{-T}^T g\li(\frac{|e^{tH_{\im p_0}}(X)|
  ^2}{\epsilon}\ri)\re p_0(e^{t H_{\im p_0}}(X))dt + O(\delta^2 \epsilon) \\ &
   \geq \frac{\delta}{O(1)} \la \re p_0\ra_{\im p_0,T}(X)- O(\delta^2 \epsilon) \geq \frac{\delta \epsilon}{O(1)}.
\end{split}
\end{equation}
Here, the last estimate is obtained using Lemma \ref{lemma_avg} and by arguing as in the region $|X|^2 \leq \frac{\epsilon}{C}.$

To summarize the discussion so far, we have shown that there exist positive constants $C>1, \widetilde{C}>1, 0 <\epsilon_0\leq 1, 0<\delta_0 \leq 1,$ such that the weight function $G_\epsilon \in C^\infty(\R^{2n};\R)$ introduced in (\ref{eq:weight_fun_defn}) satisfies for all $0<\epsilon \leq \epsilon_0, 0 < \delta \leq \delta_0,$
\begin{equation}
        \re (\widetilde{p}_0(X+i \delta H_{G_\epsilon}(X))) \geq \frac{\delta |X|^2}{\widetilde{C}}, \quad |X| \leq \epsilon^{\frac12},
    \end{equation}
\begin{equation}
        \re (\widetilde{p}_0(X+i \delta H_{G_\epsilon}(X))) \geq \frac{\delta \epsilon}{\widetilde{C}}, \quad X \in \R^{2n}, \quad \epsilon^{\frac12} \leq |X| \leq  \frac{1}{C}.
\end{equation}

Lastly, we check that the elliptic estimate (\ref{eq:elliptic_exterior_final}) holds in the unbounded region $|X|^2 \geq \frac{1}{C},$ and it is here that we use the assumption (\ref{ass_dyn_2}).
Here, (\ref{eq:interp_bound}) and (\ref{eq:hol_taylor}) yield
\begin{equation}
    \re (\widetilde{p}_0(X+i\delta H_{G_\epsilon}(X))) \geq \delta \la (\re p_0)_\epsilon\ra_{\im p_0,T}(X)-O(\delta^2\epsilon).
\end{equation}
Note that in this region, we have
\begin{equation}
    (\re p_0)_\epsilon(e^{t H_{\im p_0}}X) \sim \epsilon \re p_0(e^{tH_{\im p_0}}X), \quad |X|^2 \geq \frac{1}{C},
\end{equation}
uniformly in $|t| \leq T$ and for $\epsilon>0$ small enough, so that for $\epsilon, \delta>0$ small enough, (\ref{ass_dyn_2}) implies
\begin{equation}
    \re \widetilde{p}_0(X+i\delta H_{G_\epsilon}(X)) \geq \frac{\delta \epsilon}{O(1)} \la \re p_0\ra_{\im p_0,T}-O(\delta^2 \epsilon) \geq \frac{\delta \epsilon}{O(1)}, \quad |X|^2 \geq \frac{1}{C}.
\end{equation}

We have now verified that with $G_\epsilon$ defined as in (\ref{eq:weight_fun_defn}), the estimates (\ref{eq:elliptic_local_final}) and (\ref{eq:elliptic_exterior_final}) hold. To conclude the proof of Proposition \ref{G_prop}, we let $\chi \in C_c^\infty(\R^{2n};[0,1])$ be such that $\chi = 1$ on a large compact set and such that $\nabla \chi$ is supported in the region where $p_0$ is elliptic, and set
\begin{equation}
    \widetilde{G}_\epsilon = \chi G_\epsilon \in C_c^\infty(\R^{2n};\R). 
\end{equation}
Then, the uniform bounds (\ref{eq:g_eps_control}) still hold for $\widetilde{G}_\epsilon,$ and the conclusion of Proposition \ref{G_prop} remains valid if $G_\epsilon$ is replaced with $\widetilde{G}_\epsilon,$ in view of (\ref{ass_elliptic}). Replacing $G_\epsilon$ by $\widetilde{G}_\epsilon,$ we may conclude the proof of Proposition \ref{G_prop}.

\section{FBI transforms and Bargmann spaces}\label{fbi_and_manifolds}
In this section, we recall the basic techniques associated to Bargmann spaces and FBI transforms (for background, see \cite[Chapter 12]{sjostrand_resonance_lectures}, \cite[Chapter 13]{Zworski2012}, \cite{almost_holomorphic_ref}) and introduce suitable I-Lagrangian submanifolds of the complexified phase space $\C^{2n},$ associated to the exponential weight $G_\epsilon$ introduced in Proposition \ref{G_prop}. Throughout the section, we let $x \in \C^n$ be a complex variable and write
\begin{equation}
   \partial_x := \frac12 (\partial_{\re x}-i \partial_{\im x})
\end{equation}
for the holomorphic derivative with respect to $x.$

Let $\phi=\phi(x,y): \C^n \times \C^n \to \C$ be a holomorphic quadratic form satisfying the conditions
\begin{equation}\label{eq:phase_fun_reqs}
    \im \partial_{yy}^2 \phi >0, \quad \det \partial^2_{xy} \phi \not =0.
\end{equation}
Associated to $\phi$ is the semiclassical Fourier-Bros-Iagolnitzer (FBI) transform 
\begin{equation}\label{eq:fbi_defn}
    (T_\phi u)(x):=c_{\phi,n} h^{-\frac{3n}{4}}\int_{\R^n} e^{\frac{i}{h} \phi(x,y)} u(y)dy,
\end{equation}
where $c_{\phi,n}>0$ is a normalization constant chosen so that the map
$T_\phi$ is unitary,
\begin{equation}
    T_\phi: L^2(\R^n) \to H_{\Phi}(\C^n).
\end{equation}
Here, the Bargmann space $H_\Phi(\C^n)$ is defined as 
\begin{equation}
    H_\Phi(\C^n) :=H(\C^n) \cap L^2_\Phi(\C^n),
\end{equation} 
where $L^2_\Phi(\C^n)$ is the exponentially weighted $L^2$--space
\begin{equation}
    L^2_{\Phi}(\C^n):= L^2(\C^n,e^{-\frac{2\Phi}{h}}d\mu), \quad \mu := \text{Lebesgue measure on }\C^n,
\end{equation}
and 
\begin{equation}\label{eq:phi_defn}
    \Phi(x):=\sup_{y \in \R^n} (-\im \phi(x,y)).
\end{equation}
We note that the quadratic form $\Phi$ is strictly plurisubharmonic (see \cite{almost_holomorphic_ref}).

Associated to the map $T_\phi$ is a complex-linear canonical transformation $\kappa_{T_\phi},$ implicitly given by
\begin{equation}\label{eq:canonical_defn}
    \kappa_{T_\phi}:(y,-\partial_y\phi(x,y)) \to (x,\partial_x \phi(x,y)), \quad x,y \in \C^n, 
\end{equation}
and an I-Lagrangian, R-symplectic linear subspace $\Lambda_{\Phi}$ defined by
\begin{equation}\label{eq:lambda_phi_defn}
\Lambda_{\Phi}:=\kappa_{T_\phi}(\R^{2n})=\li\{\li(x,\frac{2}{i}\partial_x \Phi(x)\ri): x \in \C^n\ri\} \subseteq \C^{2n}.
\end{equation}
Let $a \in S(1)=S(\R^{2n},1)$ and define $\widetilde{a} := a \circ \kappa_{T_\phi}^{-1}\in S(\Lambda_{\Phi},1).$ The FBI transform satisfies an exact version of Egorov's theorem, yielding
\begin{equation}\label{eq:egorov}
    T_\phi \circ a^w = \widetilde{a}^w \circ T_\phi,
\end{equation}
where $\widetilde{a}^w$ is the semiclassical Weyl quantization of $\widetilde{a}$ and is given by the contour integral
\begin{equation}\label{eq:contour_quant_defn}
    (\widetilde{a}^wu)(x):= \frac{1}{(2\pi h)^n} \iint_{\Gamma_0(x)} e^{\frac{i}{h} (x-y) \cdot \theta} \widetilde{a}\li(\frac{x+y}{2},\theta\ri) u(y)dy \wedge d\theta
\end{equation}
for the contour of integration 
\begin{equation}
    \Gamma_0(x):=\li\{\li(y,\frac{2}{i}\partial_x \Phi\li(\frac{x+y}{2}\ri)\ri): y \in \C^n\ri\}.
\end{equation}
The integral in (\ref{eq:contour_quant_defn}) converges absolutely for $u$ in a dense subset of $H_\Phi(\C^n),$ and a contour deformation argument shows that the operator $\widetilde{a}^w$ extends to a uniformly bounded map
\begin{equation}
    \widetilde{a}^w = O(1): H_\Phi(\C^n) \to H_\Phi(\C^n),
\end{equation}
see \cite[Proposition 1.2]{sjostrand_96}. We then have
\begin{equation}\label{eq:contour_deformed_quant_defn}
    (\widetilde{a}^wu)(x):= \frac{1}{(2\pi h)^n} \iint_{\Gamma_1(x)} e^{\frac{i}{h} (x-y) \cdot \theta} \widetilde{a}\li(\frac{x+y}{2},\theta\ri) u(y)dy \wedge d\theta + Ru(x)
\end{equation}
for
\begin{equation}\label{eq:deformed_contour}
    \Gamma_1(x):= \li\{\li(y,\frac{2}{i}\partial_x \Phi\li(\frac{x+y}{2}\ri)+ic (\overline{x-y})\ri): y \in \C^n\ri\}, \quad c>0,
\end{equation}
where $R=O(h^\infty): H_\Phi(\C^n) \to L^2_\Phi(\C^n).$ In (\ref{eq:contour_deformed_quant_defn}), we let $\widetilde{a} \in C^\infty(\C^{2n})$ also stand for an almost holomorphic extension of $\widetilde{a},$ bounded with all derivatives and supported in a tubular neighborhood of $\Lambda_\Phi.$

For the rest of the paper, we fix a standard FBI transform $T$ associated to the phase function $\phi(x,y) = \frac{i}{2}(x-y)^2,$ noting that the associated canonical transformation $\kappa_T,$ quadratic weight $\Phi=\Phi_0,$ and the corresponding manifold $\Lambda_{\Phi_0}$ are given by, respectively,
\begin{equation}\label{eq:fbi_standard}
\kappa_T(x,\xi) = (x-i\xi,\xi),\quad \Phi_0(x) = \frac12 |\im x|^2, \quad \Lambda_{\Phi_0}=\{(x,-\im x):x \in \C^n\}.
\end{equation}

Let $G_0,G_\epsilon$ be the weight functions defined in (\ref{eq:G_zero_defn}) and Proposition \ref{G_prop}, respectively. Following \cite{hitrik_starov_i}, we may associate to these weights the I-Lagrangian, R-symplectic manifolds
\begin{equation}
    \Lambda_{\delta,0}=\{X+i\delta H_{G_0}(X): X \in \R^{2n}\}\subseteq \C^{2n}, \quad \Lambda_{\delta,\epsilon}=\{X+i\delta H_{G_\epsilon}(X): X \in \R^{2n}\} \subseteq \C^{2n},
\end{equation}
for $0<\delta \leq \delta_0, 0<\epsilon \leq \epsilon_0.$ Here, we recall that the weight $G_0$ is quadratic and hence $\Lambda_{\delta,0}$ is real linear. The corresponding manifolds on the Bargmann side are then
\begin{equation}
\Lambda_{\Phi_{\delta,0}}:=\kappa_T(\Lambda_{G_0})=\li\{\li(x,\frac{2}{i} \partial_x \Phi_{\delta,0}\li(x\ri)\ri): x \in \C^n\ri\},
\end{equation}
\begin{equation}\label{eq:fbi_manifold_defn}
\Lambda_{\Phi_{\delta,\epsilon}}:=\kappa_T(\Lambda_{\delta,\epsilon})=\li\{\li(x,\frac{2}{i} \partial_x \Phi_{\delta,\epsilon}\li(x\ri)\ri): x \in \C^{n}\ri\},
\end{equation}
where
\begin{equation}\label{eq:phi_quad_defn}
    \Phi_{\delta,0}(x):=\text{vc}_{(y,\eta) \in \C^{n} \times \R^n} (-\im \phi(x,y)-(\im y) \cdot \eta +\delta G_0(\re y,\eta)),
\end{equation}
\begin{equation}\label{eq:phi_G_defn}
    \Phi_{\delta,\epsilon}(x):=\text{vc}_{(y,\eta) \in \C^{n} \times \R^n} (-\im \phi(x,y)-(\im y) \cdot \eta +\delta G_\epsilon(\re y,\eta)),
\end{equation}
see \cite[Section 3]{herau_sjostrand_stolk_kfp}, \cite[Proposition 2.1]{sjostrand2010}. Here we let ``vc'' stand for the critical value.

The main properties of the weight function $\Phi_{\delta,\epsilon}$ in (\ref{eq:phi_G_defn}) were established in \cite{hitrik_starov_i}, \cite[Section 3]{herau_sjostrand_stolk_kfp}, and following these works, we recall that $\Phi_{\delta,\epsilon}$ is a strictly plurisubharmonic function satisfying 
\begin{equation}\label{eq:phi_g_taylor}
    \Phi_{\delta,\epsilon}(x) = \Phi_0(x)+\delta G_\epsilon(\re x, - \im x)+O(\delta^2 \epsilon)
\end{equation}
uniformly on $\C^n.$ In particular, we get using (\ref{eq:G_der_bound}) and (\ref{eq:phi_g_taylor}) that
\begin{equation}\label{eq:phi_eps_zero_sim}
    \Phi_{\delta,\epsilon}-\Phi_0 = O(\epsilon)
\end{equation}
uniformly on $\C^n.$ Furthermore, $\Phi_{\delta,\epsilon}-\Phi_0$ is compactly supported and
\begin{equation}\label{eq:phi_g_der_bounds}
\nabla(\Phi_{\delta,\epsilon} -\Phi_0) = O(\delta\epsilon^{\frac12}), \quad \nabla^2\Phi_{\delta,\epsilon} \in L^\infty(\C^n),
\end{equation}
uniformly with respect to the parameters $\delta,\epsilon.$ Similarly, the function $\Phi_{\delta,0}$ in (\ref{eq:phi_quad_defn}) is a strictly plurisubharmonic quadratic form on $\C^n$ such that 
\begin{equation}
    \Phi_{\delta,0}(x)=\Phi_0(x)+\delta G_0(\re x, -\im x)+O(\delta^2 |x|^2).
\end{equation}

Throughout the rest of this section, we shall assume that $\epsilon = Ah,$ where $A \geq 1$ is a large constant. It follows from (\ref{eq:phi_eps_zero_sim}) that the norms in the exponentially weighted spaces $H_{\Phi_0}(\C^n)$ and $H_{\Phi_{\delta,\epsilon}}(\C^n)$ are equivalent, uniformly as $h \to 0^+,$ for each fixed $A.$ Carrying out an additional contour deformation in (\ref{eq:contour_deformed_quant_defn}), as in \cite{herau_sjostrand_stolk_kfp}, and using (\ref{eq:phi_g_taylor}), (\ref{eq:phi_g_der_bounds}), we obtain an operator
\begin{equation}\label{eq:a_deformed_1}
    \widetilde{a}^w: H_{\Phi_{\delta,\epsilon}}(\C^n) \to H_{\Phi_{\delta,\epsilon}}(\C^n),
\end{equation}
given by 
\begin{equation}\label{eq:a_deformed_2}
    (\widetilde{a}^wu)(x) := \frac{1}{(2\pi h)^n} \iint_{\widetilde{\
    \Gamma}_1(x)} e^{\frac{i}{h} (x-y) \cdot \theta} \widetilde{a}\li(\frac{x+y}{2},\theta\ri) u(y)dy \wedge d\theta +R_1 u(x),
\end{equation}
where 
\begin{equation}\label{eq:deformed_contour_defn}
    \widetilde{\Gamma}_1(x):= \li\{\li(y,\frac{2}{i}\partial_x \Phi_{\delta,\epsilon}\li(\frac{x+y}{2}\ri)+ic (\overline{x-y})\ri): y \in \C^n\ri\}, \quad c>0,
\end{equation}
and $R_1=O_A(h^\infty): H_{\Phi_{\delta,\epsilon}}(\C^n) \to L^2_{\Phi_{\delta,\epsilon}}(\C^n).$ Using that $\nabla^2 \Phi_{\delta,\epsilon} \in L^\infty(\C^n)$ uniformly together with Schur's lemma, we see that the realization
\begin{equation}            (\widetilde{a}^w_{\widetilde{\Gamma}_1}u)(x) := \frac{1}{(2\pi h)^n} \iint_{\widetilde{\
    \Gamma}_1(x)} e^{\frac{i}{h} (x-y) \cdot \theta} \widetilde{a}\li(\frac{x+y}{2},\theta\ri) u(y)dy \wedge d\theta
\end{equation}
satisfies
\begin{equation}
\widetilde{a}_{\widetilde{\Gamma}_1}^w=O(1): H_{\Phi_{\delta,\epsilon}}(\C^n) \to L^2_{\Phi_{\delta,\epsilon}}(\C^n),
\end{equation}
uniformly with respect to the parameters $\delta,\epsilon,$ provided that the constant $c>0$ in (\ref{eq:deformed_contour_defn}) is taken sufficiently large.

\section{\texorpdfstring{Local resolvent estimates away from the critical set and proof of Theorem \ref{thm0}}{Local resolvent estimates away from the critical set and proof of Theorem 1.1}}\label{ext_estimate}
Let $p=p(x,\xi;h) \in S(\R^{2n},1)$ be such that 
\begin{equation}\label{eq:ext_symbol}
    p(x,\xi;h) \sim \sum_{k=0}^\infty h^k p_k(x,\xi)
\end{equation}
in $S(\R^{2n},1),$ and assume that the leading symbol $p_0$ satisfies the assumptions \textup{(\ref{ass_pos}), (\ref{ass_elliptic}), (\ref{ass_morse}), (\ref{ass_dyn_1}),} and \textup{(\ref{ass_dyn_2})}, so that Proposition \ref{G_prop} holds. Replacing $p_0$ by $p_0-p_0(0),$ where $p_0(0) \in i \R,$ we may assume that $p_0(0)=0.$ Let us also define $\widetilde{p}=p \circ \kappa^{-1} \in S(\Lambda_{\Phi_0},1),$ where $\kappa$ and $\Lambda_{\Phi_0}$ are given in (\ref{eq:fbi_standard}). We then have 
\begin{equation}\label{eq:fbi_p_asymptotic}
    \widetilde{p} \sim \sum_{k=0}^\infty h^k \widetilde{p}_k, \quad \widetilde{p}_k:=p_k \circ \kappa^{-1}, \quad k \geq 0,
\end{equation}
in $S(\Lambda_{\Phi_0},1).$ In this section, we shall also use the notation $\widetilde{p}_k \in C^\infty(\C^{2n})$ to denote an almost holomorphic extension of $\widetilde{p}_k \in S(\Lambda_{\Phi_0},1),$ bounded with all derivatives and supported in a tubular neighborhood of $\Lambda_{\Phi_0}.$

Let us recall the strictly plurisubharmonic weight function $\Phi_{\delta,\epsilon} \in C^\infty(\C^n)$ introduced in (\ref{eq:phi_G_defn}). Here, $\delta>0$ is sufficiently small and fixed and in what follows, we take $\epsilon=Ah,$ where $A>0$ is a sufficiently large constant to be chosen. It follows from (\ref{eq:a_deformed_1}), (\ref{eq:a_deformed_2}), and (\ref{eq:deformed_contour_defn}) that we may neglect the error term $R_1=O_A(h^\infty): H_{\Phi_{\delta,\epsilon}}(\C^n) \to L^2_{\Phi_{\delta,\epsilon}}(\C^n)$ and realize the operator
\begin{equation}
\widetilde{p}^w=\widetilde{p}^w(x,hD;h):H_{\Phi_{\delta,\epsilon}}(\C^n) \to L^2_{\Phi_{\delta,\epsilon}}(\C^n)
\end{equation}
using the contour (\ref{eq:deformed_contour_defn}). Recalling also the I-Lagrangian, R-symplectic manifold $\Lambda_{\Phi_{\delta,\epsilon}}$ defined in (\ref{eq:fbi_manifold_defn}), we obtain from (\ref{eq:elliptic_exterior_final}) and (\ref{eq:fbi_p_asymptotic}) that
\begin{equation}\label{eq:weak_elliptic_bound}
    \re \widetilde{p}_0\li(x,\frac2i \partial_x \Phi_{\delta,\epsilon}(x)\ri) \geq \frac{\delta \epsilon}{C}, \quad |x| \geq \epsilon^{\frac12}.
\end{equation}

In order to exploit the weakly elliptic bound (\ref{eq:weak_elliptic_bound}) when deriving estimates in the exterior region $|x| \geq \epsilon^{\frac12},$ following \cite{herau_sjostrand_stolk_kfp}, \cite{HitrikPravdaStarov2009}, we shall use a rescaling argument and set
 \begin{equation}\label{eq:change_of_vars_exterior}
    x:=\epsilon^{\frac12} \widetilde{x}.
 \end{equation}
Let us also set 
\begin{equation}\label{eq:h_exterior_rescaling}
    \widetilde{h}:=\frac{h}{\epsilon}=\frac{1}{A}
\end{equation}
and
\begin{equation}\label{eq:phi_tilde_epsilon_defn}
     \quad \widetilde{\Phi}_{\delta,\epsilon}(\widetilde{x}):=\frac{1}{\epsilon}\Phi_{\delta,\epsilon}(\epsilon^\frac12\widetilde{x}),
\end{equation}
and let us observe that $\nabla^2 \widetilde{\Phi}_{\delta,\epsilon} \in L^\infty(\C^n)$ uniformly with respect to $0<\delta \leq \delta_0, 0<\epsilon \leq \epsilon_0$ in view of (\ref{eq:phi_g_der_bounds}).
Associated to the change of variables (\ref{eq:change_of_vars_exterior})
is the rescaling map
\begin{equation}
     \quad Uu(\widetilde{x}) = \epsilon^{\frac{n}{2}} u(\epsilon^{\frac12}\widetilde{x}),
\end{equation}
which is unitary, 
\begin{equation}
    U: H_{\Phi_{\delta,\epsilon}}(\C^n) = H_{\Phi_{\delta,\epsilon},h}(\C^n) \to H_{\widetilde{\Phi}_{\delta,\epsilon},\widetilde{h}}(\C^n).
\end{equation}

Realizing the operator $\widetilde{p}_0^w(x,hD): H_{\Phi_{\delta,\epsilon}}(\C^n) \to L^2_{\Phi_{\delta,\epsilon}}(\C^n)$ using the contour in (\ref{eq:deformed_contour_defn}), we make the change of variables $x=\epsilon^{\frac12}\widetilde{x}, y=\epsilon^{\frac12} \widetilde{y}, \theta = \epsilon^{\frac12}\widetilde{\theta}$ to obtain
\begin{equation}\label{p_tilde_quant_defn}
    \begin{split}
        \frac{1}{\epsilon}\widetilde{p}_0^w(x,hD) u(x) & = \frac{1}{(2\pi h)^n} \iint_{\Gamma(x)} e^{\frac{i}{h} (x-y) \cdot \theta}\li(\frac1\epsilon\widetilde{p}_0\li(\frac{x+y}{2},\theta\ri)\ri) u(y)dy \wedge d\theta \\
        & = \frac{1}{(2\pi (\frac{h}{\epsilon}))^n} \iint_{\widetilde{\Gamma}(\widetilde{x})} e^{\frac{i\epsilon}{h}(\widetilde{x}-\widetilde{y}) \cdot \widetilde{\theta}} \li(\frac1\epsilon\widetilde{p}_0\li(\epsilon^{\frac12}\li(\frac{\widetilde{x}+\widetilde{y}}{2}, \widetilde{\theta}\ri)\ri)\ri) \widetilde{u}(\widetilde{y})d\widetilde{y}\wedge d\widetilde{\theta} \\
        & = \frac{1}{(2\pi \widetilde{h})^n} \iint_{\widetilde{\Gamma}(\widetilde{x})} e^{\frac{i}{\widetilde{h}}(\widetilde{x}-\widetilde{y}) \cdot \widetilde{\theta}} \widetilde{p}_{0,\epsilon}\li(\frac{\widetilde{x}+\widetilde{y}}{2}, \widetilde{\theta}\ri) \widetilde{u}(\widetilde{y})d\widetilde{y}\wedge d\widetilde{\theta} \\
    \end{split}
\end{equation}
for $\widetilde{u}(\widetilde{x})=u(\epsilon^{\frac12} \widetilde{x})$ and
\begin{equation}\label{eq:rescaled_symbol}
    \widetilde{p}_{0,\epsilon}(\widetilde{x},\widetilde{\xi}):= \frac{\widetilde{p}_0(\epsilon^\frac12 (\widetilde{x},\widetilde{\xi}))} {\epsilon}, \quad (\widetilde{x},\widetilde{\xi}) \in \C^{2n},
\end{equation}
given the contours
\begin{equation}
    \begin{split}
        \Gamma(x) & :=\li\{\li(y,\frac{2}{i}\partial_x \Phi_{\delta,\epsilon}\li(\frac{x+y}{2}\ri)+ic\overline{(x-y)}\ri): y \in \C^n\ri\}, \\
        \widetilde{\Gamma}(\widetilde{x})&:=\li\{\li(\widetilde{y},\frac{2}{i}\partial_{\widetilde{x}} \widetilde{\Phi}_{\delta,\epsilon}\li(\frac{\widetilde{x}+\widetilde{y}}{2}\ri)+ic\overline{(\widetilde{x}-\widetilde{y})}\ri): \widetilde{y} \in \C^n\ri\}.
    \end{split}
\end{equation}
It follows that 
\begin{equation}\label{eq:p_fbi_side_rescaled}
    U\li(\frac{\widetilde{p}^w_0(x,hD)}{\epsilon}\ri)U^{-1} = \widetilde{p}^w_{0,\epsilon}(\widetilde{x},\widetilde{h}D).
\end{equation}
The rescaled symbol $\widetilde{p}_{0,\epsilon}$ satisfies
\begin{equation}
 \partial^\alpha\widetilde{p}_{0,\epsilon} \in L^\infty(\C^{2n}), \quad |\alpha| \geq 2,
\end{equation}
uniformly in $0<\epsilon \leq \epsilon_0,$ and since $\widetilde{p}_0$ vanishes to the second order at $(x,\xi) =(0,0),$ we conclude that \begin{equation}
    \widetilde{p}_{0,\epsilon}(\widetilde{x},\widetilde{\xi}) = O(1) \la (\widetilde{x},\widetilde{\xi})\ra^2, \quad (\widetilde{x},\widetilde{\xi}) \in \C^{2n},
\end{equation}
uniformly in $\epsilon.$ In view of (\ref{eq:weak_elliptic_bound}) and (\ref{eq:phi_tilde_epsilon_defn}), we therefore have the elliptic bound
\begin{equation}\label{eq:elliptic_rescaled}
  \re \widetilde{p}_{0,\epsilon}\li(\widetilde{x},\frac2i \partial_{\widetilde{x}}\widetilde{\Phi}_{\delta,\epsilon} (\widetilde{x})\ri) \geq \frac{\delta}{C}, \quad |\widetilde{x}| \geq 1, \quad \widetilde{x} \in \C^n.
\end{equation}

The discussion above shows that we are in the position to apply the following quantization-\\multiplication formula to the $\widetilde{h}$--pseudodifferential operator $\widetilde{p}^w_{0,\epsilon}=\widetilde{p}^w_{0,\epsilon}(\widetilde{x},\widetilde{h}D),$ established in \\ \cite[Proposition 4]{hitrik_starov_i}:
\begin{prop}\label{prop_fbi}
    Let $\psi \in C^\infty(\C^n)$ be such that $\nabla \psi \in C_c^\infty(\C^n).$ Then, we have
    \begin{equation}\label{eq:prop_fbi}
        \langle \psi \widetilde{p}^w_{0,\epsilon} u,v \rangle_{\widetilde{\Phi}_{\delta,\epsilon},\widetilde{h}}= \int_{\C^n} \psi(\widetilde{x}) \widetilde{p}_{0,\epsilon}\li(\widetilde{x},\frac2i \partial_{\widetilde{x}}\widetilde{\Phi}_{\delta,\epsilon}(\widetilde{x})\ri)u(\widetilde{x})\overline{v(\widetilde{x})} e^{-\frac{2\widetilde{\Phi}_{\delta,\epsilon}(\widetilde{x})}{\widetilde{h}}} d\mu(\widetilde{x}) + O(\widetilde{h})\|u\|_{\widetilde{\Phi}_{\delta,\epsilon},\widetilde{h}}\|v\|_{\widetilde{\Phi}_{\delta,\epsilon},\widetilde{h}}
    \end{equation}
    for all $u,v \in H_{\widetilde{\Phi}_{\delta,\epsilon},\widetilde{h}}(\C^n).$
\end{prop}

Now, let $\chi \in C^\infty(\C^n;[0,1])$ be such that $\text{supp }\chi \subseteq \{\widetilde{x} \in \C^n: |\widetilde{x}|\geq 1\}$ and $\chi=1$ in a neighborhood of infinity. Applying Proposition \ref{prop_fbi} with $\chi^2$, we get 
\begin{equation}\label{eq:quant_mult_ext_apply}
\begin{split}
     \la \chi^2 (\widetilde{p}_{0,\epsilon}^w-\widetilde{h}z)\widetilde{u},\widetilde{u} \ra_{\widetilde{\Phi}_{\delta,\epsilon},\widetilde{h}} = \int_{\C^n} \chi^2(\widetilde{x}) \widetilde{p}_{0,\epsilon}\li(\widetilde{x},\frac2i\partial_{\widetilde{x}}\widetilde{\Phi}_{\delta,\epsilon}(\widetilde{x})\ri) & |\widetilde{u}(\widetilde{x})|^2 e^{-\frac{2\widetilde{\Phi}_{\delta,\epsilon}(\widetilde{x})}{\widetilde{h}}} d\mu(\widetilde{x}) \\ &- \widetilde{h}z \|\chi \widetilde{u}\|^2_{\widetilde{\Phi}_{\delta,\epsilon},\widetilde{h}}+ O(\widetilde{h})\|\widetilde{u}\|_{\widetilde{\Phi}_{\delta,\epsilon},\widetilde{h}}^2. 
\end{split}
\end{equation}
Assuming that the spectral parameter $z \in \C$ is such that $\re z \leq B$ for some fixed $B>0,$ we take real parts in (\ref{eq:quant_mult_ext_apply}) and use (\ref{eq:elliptic_rescaled}) to get
\begin{equation}
  \int_{\C^n} \chi^2(\widetilde{x}) |\widetilde{u}(\widetilde{x})|^2 e^{-\frac{2\widetilde{\Phi}_{\delta,\epsilon}(\widetilde{x})}{\widetilde{h}}} d\mu(\widetilde{x}) \leq   O(1)\re\la \chi^2 (\widetilde{p}_{0,\epsilon}^w-\widetilde{h}z)\widetilde{u},\widetilde{u} \ra_{\widetilde{\Phi}_{\delta,\epsilon},\widetilde{h}}  +O_B(\widetilde{h})\|\widetilde{u}\|_{\widetilde{\Phi}_{\delta,\epsilon},\widetilde{h}}^2.
\end{equation}
Using the Cauchy-Schwarz inequality, we thus obtain
\begin{equation}\label{eq:exterior_after_cauchy_schwarz}
     \int_{\C^n} \chi^2(\widetilde{x}) |\widetilde{u}(\widetilde{x})|^2 e^{-\frac{2\widetilde{\Phi}_{\delta,\epsilon}(\widetilde{x})}{\widetilde{h}}} d\mu(\widetilde{x}) \leq O(1)\|(\widetilde{p}^w_{0,\epsilon}-\widetilde{h}z)\widetilde{u}\|_{\widetilde{\Phi}_{\delta,\epsilon},\widetilde{h}}\|\widetilde{u}\|_{\widetilde{\Phi}_{\delta,\epsilon},\widetilde{h}}  + O_B(\widetilde{h})\|\widetilde{u}\|_{\widetilde{\Phi}_{\delta,\epsilon},\widetilde{h}}^2.
\end{equation}
Writing $\widetilde{u}=Uu$ for $u \in H_{\Phi_{\delta,\epsilon}}(\C^n),$ we get in view of (\ref{eq:change_of_vars_exterior}), (\ref{eq:p_fbi_side_rescaled}), (\ref{eq:exterior_after_cauchy_schwarz}) that
\begin{equation}\label{eq:re_elliptic_exterior_final}
    \epsilon \int_{\C^n} \chi^2\li(\frac{x}{\sqrt{\epsilon}}\ri)|u(x)|^2 e^{-\frac{2\Phi_{\delta,\epsilon}(x)}{h}}d\mu(x) \leq O(1) \|(\widetilde{p}_0^w(x,hD)-hz)u\|_{\Phi_{\delta,\epsilon}} \|u\|_{\Phi_{\delta,\epsilon}} + O_B(h) \|u\|_{\Phi_{\delta,\epsilon}}^2
\end{equation}
for all $u \in H_{\Phi_{\delta,\epsilon}}(\C^n).$ The realization of the operator
\begin{equation}
    \widetilde{p}^w(x,hD;h)-\widetilde{p}_0^w(x,hD): H_{\Phi_{\delta,\epsilon}}(\C^n) \to L^2_{\Phi_{\delta,\epsilon}}(\C^n)
\end{equation}
using the contour (\ref{eq:deformed_contour_defn}) satisfies
\begin{equation}\label{eq:p_p0_similar}
    \widetilde{p}^w(x,hD;h)-\widetilde{p}_0^w(x,hD)=O(h): H_{\Phi_{\delta,\epsilon}}(\C^n) \to L^2_{\Phi_{\delta,\epsilon}}(\C^n)
\end{equation}
uniformly with respect to $\epsilon.$ Consequently, using (\ref{eq:re_elliptic_exterior_final}), (\ref{eq:p_p0_similar}), and reintroducing the negligible error term in (\ref{eq:a_deformed_2}), we get that
\begin{equation}\label{eq:exterior_final}
    \epsilon \int_{\C^n} \chi^2\li(\frac{x}{\sqrt{\epsilon}}\ri)|u(x)|^2 e^{-\frac{2\Phi_{\delta,\epsilon}(x)}{h}}d\mu(x) \leq O(1) \|(\widetilde{p}^w-hz)u\|_{\Phi_{\delta,\epsilon}} \|u\|_{\Phi_{\delta,\epsilon}} + (O_B(h)+O_A(h^\infty)) \|u\|_{\Phi_{\delta,\epsilon}}^2.
\end{equation}
The discussion above may be summarized in the following result.

\begin{prop}\label{prop_ext}
    Let $p(x,\xi;h) \sim p_0(x,\xi) + hp_1(x,\xi)+\cdots$ in $S(\R^{2n},1),$ and assume that $p_0$ satisfies \textup{(\ref{ass_pos}), (\ref{ass_elliptic}), (\ref{ass_morse}), (\ref{ass_dyn_1}),} and \textup{(\ref{ass_dyn_2}).} We let 
    \begin{equation}
        \widetilde{p}^w = T \circ p^w(x,hD;h) \circ T^{-1},
    \end{equation}
    where $T$ is a semiclassical metaplectic FBI-Bargmann transform given in \textup{(\ref{eq:fbi_defn}),} with $\phi(x,y) = \frac{i}{2}(x-y)^2.$ Let $\epsilon=Ah$ for $A \geq 1,$ let $\Phi_{\delta,\epsilon} \in C^\infty(\C^n)$ be defined in \textup{(\ref{eq:phi_G_defn}),} and let $B>0.$ Let also $\chi \in C^\infty(\C^n;[0,1])$ be such that $\text{supp }\chi \subseteq \{\widetilde{x} \in \C^n:|\widetilde{x}| \geq 1\}$ and $\chi = 1$ in a neighborhood of infinity. Then, we have that for all $u \in H_{\Phi_{\delta,\epsilon}}(\C^n),$ all $z \in \C$ such that $\re z \leq B,$ and all $0<h \leq h_0,$
    \begin{equation}\label{eq:exterior_final_2}
    \begin{split}
    \epsilon \int_{\C^n} \chi^2\li(\frac{x}{\sqrt{\epsilon}}\ri)&|u(x)|^2 e^{-\frac{2\Phi_{\delta,\epsilon}(x)}{h}}d\mu(x) \\ & \leq O(1) \|(\widetilde{p}^w-p_0(0)-hz)u\|_{\Phi_{\delta,\epsilon}} \|u\|_{\Phi_{\delta,\epsilon}} + (O_B(h)+O_A(h^\infty)) \|u\|_{\Phi_{\delta,\epsilon}}^2.
    \end{split}
\end{equation}
\end{prop} 

With Proposition \ref{prop_ext} as the starting point, we shall now turn our attention to the proof of Theorem \ref{thm0}. When doing so, we may assume that $p_0(0)=0.$ Taking the imaginary part in (\ref{eq:pre_taylor}) and using (\ref{ass_quad}), (\ref{eq:double_derivative_estimate}), (\ref{eq:G_der_bound}) gives 
\begin{equation}\label{eq:im_interior_estimate}
    \im \widetilde{p}_0(X+i\delta H_{G_\epsilon}(X))  = O(\epsilon), \quad |X| \leq O(1)\epsilon^{\frac12}, \quad X \in \R^{2n}.\\
\end{equation}
Similarly to (\ref{eq:weak_elliptic_bound}), we conclude therefore that 
\begin{equation}\label{eq:imp0_size}
   \im \widetilde{p}_0\li(x,\frac2i \partial_x \Phi_{\delta,\epsilon}(x)\ri) = O(\epsilon), \quad |x| \leq 2 \epsilon^{\frac12}. 
\end{equation}

Let $\psi \in C^\infty_c(\C^n;[0,1])$ be such that $\text{supp }\psi \subseteq \{\widetilde{x}\in \C^n: |\widetilde{x}|<2\}, \psi(\widetilde{x})=1$ for $|\widetilde{x}| \leq 1$ and $\psi^2+\chi^2=1$ on $\C^n.$ Applying Proposition \ref{prop_fbi} with $\psi^2$ and taking the imaginary part, we get that for $z \in \C$ and $\widetilde{u} \in H_{\widetilde{\Phi}_{\delta,\epsilon},\widetilde{h}}(\C^n),$ one has 
\begin{equation}\label{eq:im_interior_elliptic_estimate_1}
\begin{split}
    \im\la \psi^2 &(\widetilde{p}_{0,\epsilon}^w-\widetilde{h}z)\widetilde{u},\widetilde{u} \ra_{\widetilde{\Phi}_{\delta,\epsilon},\widetilde{h}} \\ & =\int_{\C^n} \psi^2(\widetilde{x}) \li(\im\widetilde{p}_{0,\epsilon}\li(\widetilde{x},\frac2i \partial_{\widetilde{x}}\widetilde{\Phi}_{\delta,\epsilon}(\widetilde{x})\ri)-\widetilde{h}\im z\ri) |\widetilde{u}(\widetilde{x})|^2 e^{-\frac{2\widetilde{\Phi}_{\delta,\epsilon}(\widetilde{x})}{\widetilde{h}}} d\mu(\widetilde{x}) + O(\widetilde{h})\|\widetilde{u}\|_{\widetilde{\Phi}_{\delta,\epsilon},\widetilde{h}}^2.
\end{split}
\end{equation}
Here we have 
\begin{equation}\label{eq:imp0eps_size}
   \im\widetilde{p}_{0,\epsilon}\li(\widetilde{x},\frac2i \partial_{\widetilde{x}}\widetilde{\Phi}_{\delta,\epsilon}(\widetilde{x})\ri) = O(1)
\end{equation}
on the support of $\psi$ in view of (\ref{eq:phi_tilde_epsilon_defn}), (\ref{eq:rescaled_symbol}), and (\ref{eq:imp0_size}). Assuming that 
\begin{equation}\label{eq:im_bound}
    |\im z|>D,
\end{equation}
where $D>0$ is sufficiently large depending on $A,$ we obtain therefore using (\ref{eq:im_interior_elliptic_estimate_1}), (\ref{eq:imp0eps_size}), and the Cauchy-Schwarz inequality that

\begin{equation}\label{eq:after_cauchy}
    \int_{\C^n} \psi^2(\widetilde{x}) |\widetilde{u}(\widetilde{x})|^2 e^{-\frac{2\widetilde{\Phi}_{\delta,\epsilon}(\widetilde{x})}{\widetilde{h}}}d\mu(\widetilde{x}) \leq O(1)\|(\widetilde{p}_{0,\epsilon}^w-\widetilde{h}z)\widetilde{u}\|_{\widetilde{\Phi}_{\delta,\epsilon},\widetilde{h}}\|\widetilde{u}\|_{\widetilde{\Phi}_{\delta,\epsilon},\widetilde{h}}  + O(\widetilde{h})\|\widetilde{u}\|_{\widetilde{\Phi}_{\delta,\epsilon},\widetilde{h}}^2,
\end{equation}
where we treat the cases $\im z>D$ and $\im z<-D$ separately when establishing (\ref{eq:after_cauchy}). Similarly to (\ref{eq:re_elliptic_exterior_final}), we infer from (\ref{eq:after_cauchy}) that
\begin{equation}
    \epsilon\int_{\C^n} \psi^2\li(\frac{x}{\sqrt{\epsilon}}\ri) |u(x)|^2 e^{-\frac{2\Phi_{\delta,\epsilon}(x)}{h}}d\mu(x) \leq O(1)\|(\widetilde{p}_{0}^w(x,hD)-hz)u\|_{\Phi_{\delta,\epsilon}}\|u\|_{\Phi_{\delta,\epsilon}}  + O(h)\|u\|_{\Phi_{\delta,\epsilon}}^2
\end{equation}
for all $u \in H_{\Phi_{\delta,\epsilon}}(\C^n)$ and all $z \in \C$ satisfying (\ref{eq:im_bound}). Next, arguing as in (\ref{eq:p_p0_similar}) and (\ref{eq:exterior_final}), we get
\begin{equation}\label{eq:im_elliptic_interior_final}
\begin{split}
    \epsilon \int_{\C^n} \psi^2\li(\frac{x}{\sqrt{\epsilon}}\ri)&|u(x)|^2 e^{-\frac{2\Phi_{\delta,\epsilon}(x)}{h}}d\mu(x) \\ & \leq O(1) \|(\widetilde{p}^w(x,hD;h)-hz)u\|_{\Phi_{\delta,\epsilon}} \|u\|_{\Phi_{\delta,\epsilon}} + (O(h)+O_A(h^\infty)) \|u\|_{\Phi_{\delta,\epsilon}}^2.
\end{split}
\end{equation}
Combining (\ref{eq:exterior_final}) and (\ref{eq:im_elliptic_interior_final}), we obtain that for all $z \in \C$ such that $\re z \leq B$ and $|\im z| >D,$ for some $D$ large enough depending on $A \geq 1,$ we have 
\begin{equation}
    \epsilon \int_{\C^n} |u(x)|^2 e^{-\frac{2\Phi_{\delta,\epsilon}(x)}{h}}d\mu(x) \leq O(1) \|(\widetilde{p}^w(x,hD;h)-hz)u\|_{\Phi_{\delta,\epsilon}}\|u\|_{\Phi_{\delta,\epsilon}}+(O_B(h)+O_A(h^\infty))\|u\|_{\Phi_{\delta,\epsilon}}^2.
\end{equation}
Here, we recall that $\epsilon=Ah.$ Thus,
\begin{equation}\label{eq:final_thm0_Ah}
\begin{split}
    h \int_{\C^n} |u(x)|^2 & e^{-\frac{2\Phi_{\delta,\epsilon}(x)}{h}}d\mu(x) \\ & \leq O(1) \|(\widetilde{p}^w(x,hD;h)-hz)u\|_{\Phi_{\delta,\epsilon}} \|u\|_{\Phi_{\delta,\epsilon}} +\li(O_B\li(\frac{h}{A}\ri)+O_A(h^\infty)\ri) \|u\|_{\Phi_{\delta,\epsilon}}^2,
\end{split}
\end{equation}
and writing
\begin{equation}\label{eq:final_thm0_am_gm}
    O(1)\|(\widetilde{p}^w(x,hD;h)-hz)u\|_{\Phi_{\delta,\epsilon}}\|u\|_{\Phi_{\delta,\epsilon}} \leq \frac{O(1)}{h}\|(\widetilde{p}^w(x,hD;h)-hz)u\|^2_{\Phi_{\delta,\epsilon}}+\frac{h}{2}\|u\|_{\Phi_{\delta,\epsilon}}^2,
\end{equation}
we obtain using (\ref{eq:final_thm0_Ah}), (\ref{eq:final_thm0_am_gm}),
\begin{equation}\label{eq:final_thm0_final}
     h \int_{\C^n} |u(x)|^2 e^{-\frac{2\Phi_{\delta,\epsilon}(x)}{h}}d\mu(x) \leq \frac{O(1)}{h}\|(\widetilde{p}^w(x,hD;h)-hz)u\|^2_{\Phi_{\delta,\epsilon}} +\li(O_B\li(\frac{h}{A}\ri)+O_A(h^\infty)\ri) \|u\|_{\Phi_{\delta,\epsilon}}^2.
\end{equation}
For each $B>0,$ we may choose $A \geq 1$ large enough and fixed so that the second term on the right hand side of (\ref{eq:final_thm0_final}) can be absorbed into the left hand side for all $h>0$ small enough. We therefore get that for each $B>0,$ there exists $D>0$ such that for all $h>0$ small enough, all $u \in H_{\Phi_{\delta,\epsilon}}(\C^n),$ and all $z \in \C$ such that $\re z<B, |\im z|>D,$ we have
\begin{equation}\label{eq:full_exterior_final}
    h \|u\|_{\Phi_{\delta,\epsilon}}\leq O(1) \|(\widetilde{p}^w-hz)u\|_{\Phi_{\delta,\epsilon}}.
\end{equation}
Recalling (\ref{eq:phi_eps_zero_sim}), we obtain that 
\begin{equation}\label{eq:thm0_final_resolvent}
   h \|u\|_{\Phi_0} \leq O(1)\|(\widetilde{p}^w-hz)u\|_{\Phi_0}.
\end{equation}
Using (\ref{eq:thm0_final_resolvent}), the unitarity of the FBI transform $T:L^2(\R^n) \to H_{\Phi_0}(\C^n),$ and (\ref{eq:egorov}), we obtain the desired resolvent estimate
\begin{equation}
   h\|u\|_{L^2} \leq O(1)\|(p^w(x,hD;h)-hz)u\|_{L^2},
\end{equation}
for all $u \in L^2(\R^n)$ and all $z \in \C$ such that $\re z<B,|\im z|>D.$ It follows that the operator
\begin{equation}
    p^w(x,hD;h)-hz: L^2(\R^n) \to L^2(\R^n)
\end{equation}
is injective with closed range. Recalling (\ref{ass_elliptic}), we observe that it is also Fredholm of index zero for all $h>0$ small enough depending on $B.$ We conclude that 
\begin{equation}
    p^w(x,hD;h)-hz:L^2(\R^n) \to L^2(\R^n)
\end{equation}
is invertible for all $h>0$ small enough and all such $z,$ with the inverse satisfying 
\begin{equation}
    \|(p^w(x,hD;h)-hz)^{-1}\|_{\mathcal{L}(L^2(\R^n)),L^2(\R^n))} \leq O_B\li(\frac{1}{h}\ri).
\end{equation}
This completes the proof of Theorem \ref{thm0}.
\section{A priori estimates for elliptic quadratic operators}\label{local_estimate}
In this section, as a preparation for the proof of Theorem \ref{thm1}, we shall discuss a priori estimates for semiclassical differential operators with elliptic quadratic symbols acting on Bargmann spaces. Let $\Phi$ be a strictly plurisubharmonic quadratic form on $\C^n,$ and let us set, similarly to (\ref{eq:lambda_phi_defn}),
\begin{equation}
    \Lambda_\Phi =\li\{\li(x,\frac2i \partial_x \Phi(x)\ri): x \in \C^n\ri\} \subseteq \C^{2n}.
\end{equation}
Let $\widetilde{q}:\C^{2n} \to \C$ be a holomorphic quadratic form satisfying 
\begin{equation}\label{eq:ass_quad}
        \re \widetilde{q}|_{\Lambda_{\Phi}}>0
\end{equation}
in the sense of quadratic forms. Associated to $\widetilde{q}$ is the semiclassical Weyl quantization $\widetilde{q}^w(x,hD)$ acting on $H_\Phi(\C^n)$ as an unbounded operator, given by 
\begin{equation}
    (\widetilde{q}^w(x,hD)u)(x):= \frac{1}{(2\pi h)^n} \iint_{\Gamma(x)} e^{\frac{i}{h} (x-y) \cdot \theta} \widetilde{q}\li(\frac{x+y}{2},\theta\ri) u(y)dy \wedge d\theta,
\end{equation}
with
\begin{equation}
    \Gamma(x):= \li\{\li(y,\frac{2}{i}\partial_x \Phi\li(\frac{x+y}{2}\ri)+ic (\overline{x-y})\ri): y \in \C^n\ri\}, \quad c>0.
\end{equation}
When equipped with its maximal domain, $\widetilde{q}^w(x,hD)$ becomes a closed densely defined operator on $H_\Phi(\C^n)$ with discrete spectrum (see \cite[Section 5]{herau_sjostrand_stolk_kfp}).

Our starting point is the following result, which has been established in \cite[Proposition 5.1]{herau_sjostrand_stolk_kfp}. Since the argument in \cite{herau_sjostrand_stolk_kfp} is quite brief, here we shall take an opportunity to provide some additional details. 
\begin{prop}\label{quad_bound}
    Let $K \subseteq \C$ be a compact set disjoint from the spectrum of $\widetilde{q}^w(x,D)$ on \\ $H_{\Phi,h=1}(\C^n).$ Then, there exists $C>0$ such that the estimate
    \begin{equation}\label{eq:res_2}
        \|(h+|x|^2)^{\frac12}u\|_{L^2_\Phi(\C^n)} \leq C \|(h+|x|^2)^{-\frac12}(\widetilde{q}^w(x,hD)-h\lambda)u\|_{L^2_\Phi(\C^n)}
    \end{equation}
    holds for all $0<h \leq 1,$ all $\lambda \in K,$ and all $u \in H(\C^n)$ such that $u=O_{k,h}(1)\la x\ra^{-k} e^{\frac{\Phi(x)}{h}}$ in $L^2(\C^n)$ for all $k \geq 0.$
\end{prop}
\begin{proof} 
We begin the proof by noting that there exists a holomorphic quadratic form $\phi=\phi(x,y)$ on $\C^{2n},$ satisfying (\ref{eq:phase_fun_reqs}), such that the corresponding FBI transform $T$ in (\ref{eq:fbi_defn}) is a unitary map $T:L^2(\R^n) \to H_\Phi(\C^n),$ and such that the associated complex linear canonical transformation $\kappa_T$ in (\ref{eq:canonical_defn}) satisfies 
\begin{equation}
    \kappa_T(\R^{2n}) = \Lambda_\Phi,
\end{equation}
(see \cite[Section 13.3]{Zworski2012}). 

The quadratic form $q:= \widetilde{q} \circ \kappa_T$ satisfies $\re q>0$ on $\R^{2n}$ by virtue of (\ref{eq:ass_quad}). The maximal domain
    \begin{equation}
        \mathcal{D}(q^w(x,D)) = \{u \in L^2(\R^n): q^w(x,D)u \in L^2(\R^n)\}
    \end{equation}
    of the quadratic differential operator $q^w(x,D)$ with $h=1$ is then
    \begin{equation}
        \mathcal{D}(q^w(x,D)) = \{u \in L^2(\R^n): (x^2+D^2)u \in L^2(\R^n)\}
    \end{equation}
    in view of the ellipticity of $q.$ Using the functional calculus for the self-adjoint  operator $x^2+D^2$ on $L^2(\R^n),$ let us set
    \begin{equation}
        \Lambda = \la (x,D)\ra =(1+x^2+D^2)^{\frac12}.
    \end{equation}
    Here, we recall from \cite[Theorem 1.11.1]{AST_1984__112__R1_0} that
    \begin{equation}\label{eq:lambda_pseudo}
        \Lambda^r \in \text{Op}^w(S(\la X\ra^r)), \quad r \in \R.
    \end{equation}
    We let $\mathcal{D}(\Lambda)=\Lambda^{-1}(L^2(\R^n))$ be the domain of the unbounded self-adjoint operator $\Lambda,$ equipped with the graph norm, and also introduce $\mathcal{D}(\Lambda^{-1})=\Lambda(L^2(\R^n)) \subseteq \s'(\R^n),$ the dual space of $\mathcal{D}(\Lambda),$ with duality defined via the $L^2$ scalar product.

    The ellipticity property $\re q>0$ implies that for each $\lambda \in \C,$ the operator
    \begin{equation}
        q^w(x,D)-\lambda:\mathcal{D}(\Lambda) \to \mathcal{D}(\Lambda^{-1})
    \end{equation}
    is a Fredholm operator of index $0,$ which is invertible for $\lambda \not \in \text{Spec}_{L^2(\R^n)}(q^w(x,D)).$ We conclude that there exists $C=C_K>0$ such that 
    \begin{equation}
        \|(q^w(x,D)-\lambda)^{-1}\|_{\mathcal{L}(D(\Lambda^{-1}),D(\Lambda))} \leq C, \quad \lambda \in K,
    \end{equation}
    and therefore
    \begin{equation}
        \|v\|_{\mathcal{D}(\Lambda)} \leq C \|(q^w(x,D)-\lambda)v\|_{\mathcal{D}(\Lambda^{-1})}, \quad \lambda \in K, \quad v \in \s(\R^n),
    \end{equation}
    or equivalently,
\begin{equation}\label{eq:resolvent_bound}
         \|\Lambda v\|_{L^2(\R^n)} \leq C\|\Lambda^{-1}(q^w(x,D)-\lambda)v\|_{L^2(\R^n)}, \quad \lambda \in K, \quad v \in \s(\R^n).
    \end{equation}
    Letting $u=Tv,$ where $T:L^2(\R^n) \to H_{\Phi,h=1}(\C^n)$ is a unitary FBI transform with $h=1,$ we may use the exact Egorov theorem \cite{sjostrand_resonance_lectures} to rewrite (\ref{eq:resolvent_bound}) as
    \begin{equation}\label{eq:fbi_side_est}
        \|\widetilde{\Lambda} u\|_{H_{\Phi,h=1}(\C^n)} \leq C\|\widetilde{\Lambda}^{-1}(\widetilde{q}^w(x,D)-\lambda)u\|_{H_{\Phi,h=1}(\C^n)}, \quad \lambda \in K, \quad u \in T(\s(\R^n)).
    \end{equation}
    Here, $\widetilde{\Lambda}=T\Lambda T^{-1}$ and therefore 
    \begin{equation}\label{eq:lambda_tilde_inv}
        \widetilde{\Lambda}^{-1} = T\Lambda^{-1} T^{-1} = c^w(x,D), \quad c \in S(\Lambda_\Phi, \la x\ra^{-1}),
    \end{equation}
    where we also use (\ref{eq:lambda_pseudo}) in the second equality. It follows from \cite[Proposition 12.6]{sjostrand_resonance_lectures} and (\ref{eq:lambda_tilde_inv}) that 
    \begin{equation}\label{eq:lambda_tilde_inv_bounded}
\widetilde{\Lambda}^{-1}:H_{\Phi,h=1}(\la x\ra^{-1}) \to H_{\Phi,h=1}(\C^n)
    \end{equation}
    is bounded, where
    \begin{equation}
        H_{\Phi,h=1}(\la x\ra^{-1})=H(\C^n) \cap L^2(\C^n,\la x\ra^{-2} e^{-2\Phi}d\mu).
    \end{equation}
   Then, by combining (\ref{eq:fbi_side_est}), (\ref{eq:lambda_tilde_inv_bounded}),
    and the self-adjointness of $\widetilde{\Lambda}$ on $H_{\Phi,h=1}(\C^n)$, we get 
    \begin{equation}\label{eq:order_function_est}
        \|\widetilde{\Lambda}u\|^2_{H_{\Phi,h=1}(\C^n)} = \la \widetilde{\Lambda}^2u,u\ra_{H_{\Phi,h=1}(\C^n)}  \leq C\int_{\C^n} \la x\ra^{-2}|(\widetilde{q}^w(x,D)-\lambda)u(x)|^2e^{-2\Phi(x)}d\mu(x)
    \end{equation}
    for all $\lambda \in K$ and all $u \in T(\s(\R^n)).$

    It only remains to restore the semiclassical parameter $h.$ To this end, we introduce the unitary rescaling map $U: H_{\Phi,h=1}(\C^n) \to H_{\Phi,h}(\C^n)=H_\Phi(\C^n),$ given by \begin{equation}\label{eq:rescaling}
      (Uu)(x) = \frac{1}{h^{\frac{n}{2}}} u\li(\frac{x}{h^\frac12}\ri), \quad 0<h \leq 1.
  \end{equation}
  We have that 
      \begin{equation}
         U \widetilde{q}^w(x,D) U^{-1}=\frac{1}{h} \widetilde{q}^w(x,hD),
    \end{equation}
   so making the change of variables $x=\frac{y}{h^{\frac12}},$ we see that the integral on the right-hand side of (\ref{eq:order_function_est}) becomes
   \begin{equation}\label{eq:int_rewrite}
       \int_{\C^n} \frac{h}{h+|y|^2}\frac{|(\widetilde{q}^w(y,hD)-h\lambda)Uu(y)|^2}{h^2}e^{-\frac{2\Phi(y)}{h}}d\mu(y).
   \end{equation}
   Using (\ref{eq:order_function_est}),  (\ref{eq:int_rewrite}), and writing $u$ instead of $Uu,$ we obtain
    \begin{equation}\label{eq:rescaled_est}
   \la \widetilde{\Lambda}_h^2u,u\ra_{H_\Phi(\C^n)} \leq \frac{C}{h} \int_{\C^n} \frac{1}{h+|x|^2}|(\widetilde{q}^w(x,hD)-h\lambda)u(x)|^2e^{-\frac{2\Phi(x)}{h}}d\mu(x)
    \end{equation}
    for all $\lambda \in K$ and $u \in H(\C^n)$ such that $u=O_{k,h}(1)\la x\ra^{-k} e^{\frac{\Phi(x)}{h}}$ in $L^2(\C^n)$ for all $k \geq 0.$ Here, $\widetilde{\Lambda}_h=U\widetilde{\Lambda}U^{-1},$ and we rely on the characterization of the space $T(\s(\R^n))$ given in \cite[Chapter 12]{sjostrand_resonance_lectures}. 
    We have that 
    \begin{equation}\label{eq:lambda_tilde_h_defn}
        \widetilde{\Lambda}_h^2 = U\widetilde{\Lambda}^2 U^{-1}= 1+ \frac{1}{h} q_0^w(x,hD),
    \end{equation}
    where $q_0$ is a holomorphic quadratic form on $\C^{2n}$ such that $q_0|_{\Lambda_\Phi}$ is real positive definite and $q_0^w(x,hD) \geq 0$ in the sense of self-adjoint operators on $H_\Phi(\C^n).$
    
    Using (\ref{eq:rescaled_est}) and (\ref{eq:lambda_tilde_h_defn}), we get that 
    \begin{equation}\label{eq:id_est}
    h\|u\|_{H_\Phi(\C^n)}^2 +\la q_{0}^w(x,hD) u,u\ra_{H_\Phi(\C^n)} \leq C \int_{\C^n} \frac{1}{h+|x|^2}|(\widetilde{q}^w(x,hD)-h\lambda)u(x)|^2e^{-\frac{2\Phi(x)}{h}}d\mu(x).
    \end{equation}
    Here, an application of the the quantization-multiplication formula for quadratic operators \cite{HitrikSjostrandViola2013} gives
    \begin{equation}\label{eq:apply_qm}
        \la q_0^w(x,hD)u,u\ra_{H_\Phi(\C^n)} = \int_{\C^n} q_0\li(x,\frac2i \partial_x\Phi(x)\ri) |u(x)|^2 e^{-\frac{2\Phi(x)}{h}} d\mu(x)+O(h)\|u\|^2_{H_\Phi(\C^n)},
        \end{equation}
        where 
        \begin{equation}\label{eq:q0_est}
            q_0\li(x,\frac2i \partial_x\Phi(x)\ri) \sim |x|^2, \quad x \in \C^n.
        \end{equation}
    Thus, by combining (\ref{eq:id_est}),  (\ref{eq:apply_qm}), and (\ref{eq:q0_est}), we obtain 
    \begin{equation}
       h\|u\|^2_{H_\Phi(\C^n)} + \int_{\C^n} |x|^2|u(x)|^2 e^{-\frac{2\Phi(x)}{h}}d\mu(x) \leq C \int_{\C^n} \frac{1}{h+|x|^2} |(\widetilde{q}^w(x,hD)-h\lambda)u(x)|^2e^{-\frac{2\Phi(x)}{h}}d\mu(x)
    \end{equation}
    for all $\lambda \in K$ and $u \in H(\C^n)$ such that $u=O_{k,h}(1)\la x\ra^{-k} e^{\frac{\Phi(x)}{h}}$ in $L^2(\C^n)$ for all $k \geq 0.$ This completes the proof.
\end{proof}

The following result, providing a localization of the a priori estimate (\ref{eq:res_2}), has been established in \cite[Proposition 5.2]{herau_sjostrand_stolk_kfp}: 
\begin{lemma}\label{lemma_quad_bound_cutoff}
    Let $\chi_0 \in C^\infty_c(\C^n)$ be such that $\chi_0 = 1$ in a neighborhood of the origin, and let $K \subseteq \C$ be a compact set disjoint from the spectrum of $\widetilde{q}^w(x,D)$ on $H_{\Phi,h=1}(\C^n).$ Then, for all $\lambda \in K$ and $u \in H(\C^n),$ we have
    \begin{equation}\label{eq:lemma_quad_bound_cutoff}
        \|(h+|x|^2)^{\frac12}\chi_0u\|_{L^2_\Phi(\C^n)} \leq O(1) \|(h+|x|^2)^{-\frac12}\chi_0(\widetilde{q}^w(x,hD)-h\lambda)u\|_{L^2_\Phi(\C^n)}+O(h^{\frac12})\|1_Ru\|_{L^2_\Phi(\C^n)}.
    \end{equation}
    Here, $R \subseteq \C^n$ is a fixed neighborhood of $\text{supp}(\nabla \chi_0)$  and $1_R$ denotes the characteristic function of $R.$
\end{lemma}

As the local estimate (\ref{eq:lemma_quad_bound_cutoff}) ultimately needs to be valid in a $O(\epsilon^{\frac12})$--neighborhood of the origin, we apply another unitary rescaling map. For $0<h\leq\epsilon \leq 1,$ we introduce a new semiclassical parameter
\begin{equation}
    \widetilde{h}:=\frac{h}{\epsilon }\leq 1
\end{equation}
and note that the map
$V: H_\Phi(\C^n)=H_{\Phi,h}(\C^n) \to H_{\Phi,\widetilde{h}}(\C^n)$ given by
\begin{equation}
    (Vu)(\widetilde{x}) = \epsilon^{\frac{n}{2}}u\li(\epsilon^{\frac12}\widetilde{x}\ri)
\end{equation}
is unitary. We have 
\begin{equation}\label{eq:V_unitary}
    V\widetilde{q}^w(x,hD)V^{-1} = \epsilon \widetilde{q}^w(\widetilde{x},\widetilde{h}D),
\end{equation}
and applying (\ref{eq:lemma_quad_bound_cutoff}) with $\widetilde{h}$ in place of $h,$ we get
\begin{equation}\label{eq:quad_bound_rescaled_h}
\begin{split}
    \|(\widetilde{h}+|\widetilde{x}|^2)^{\frac12} \chi_0u\|_{L^2_{\Phi,\widetilde{h}}(\C^n)} & \leq O(1) \|(\widetilde{h}+|\widetilde{x}|^2)^{-\frac12}\chi_0(\widetilde{q}^w(\widetilde{x},\widetilde{h}D) -\widetilde{h}\lambda)u\|_{L^2_{\Phi,\widetilde{h}}(\C^n)}\\ &+O(\widetilde{h}^{\frac12})\|1_Ru\|_{L^2_{\Phi,\widetilde{h}}(\C^n)}.
    \end{split}
\end{equation}
Replacing $u$ by $Vu$ in (\ref{eq:quad_bound_rescaled_h}) and making the change of variables $\widetilde{x} = \frac{x}{\epsilon^{\frac12}},$ we obtain using (\ref{eq:V_unitary}) that
\begin{equation}\label{eq:localized_interior_estimate}
\begin{split}
    \|(h+|x|^2)^{\frac12} \chi_0\li(\frac{x}{\sqrt{\epsilon}}\ri)u\|_{L^2_{\Phi}(\C^n)} & \leq O(1) \|(h+|x|^2)^{-\frac12}\chi_0\li(\frac{x}{\sqrt{\epsilon}}\ri)(\widetilde{q}^w(x,hD)-h\lambda)u\|_{L^2_{\Phi}(\C^n)}\\ & +O( h^{\frac12})\|1_R\li(\frac{x}{\sqrt{\epsilon}}\ri)u\|_{L^2_{\Phi}(\C^n)},
\end{split}
\end{equation}
which holds for all $\lambda \in K$ and $u \in H(\C^n).$

In our applications, the estimate (\ref{eq:localized_interior_estimate}) will be applied to the holomorphic quadratic form $\widetilde{q}=q \circ \kappa_T^{-1},$ where $q$ is defined in (\ref{ass_quad}) and $\kappa_T$ is given in (\ref{eq:fbi_standard}). Furthermore, we shall take $\Phi = \Phi_{\delta,0}$ in (\ref{eq:phi_quad_defn}) for $\delta>0$ sufficiently small but fixed. We only need to check that (\ref{eq:ass_quad}) holds, or equivalently, that $\re q|_{\Lambda_{\delta,0}}>0$ in the sense of quadratic forms. To this end, recall that the quadratic form $G_0,$ defined in (\ref{eq:G_zero_defn}), satisfies
\begin{equation}\label{eq:avg_quad}
    H_{\im q}G_0 = \la \re q \ra_{\im q, T}-\re q.
\end{equation}
An application of Taylor's formula and (\ref{eq:avg_quad}) then yield
    \begin{equation}
    \begin{split}
        \re q(X+i\delta H_{G_0}(X)) 
        & = \re q(X)+\delta H_{\im q}G_0(X) + O(\delta^2|X|^2) \\
        & = (1-\delta)\re q(X)+\delta \la \re q\ra_{\im q,T}(X)+O(\delta^2|X|^2) \geq \frac{\delta}{C} |X|^2
    \end{split}
    \end{equation}
    for $X \in \R^{2n}$ and $\delta>0$ sufficiently small, in view of (\ref{ass_dyn_1}). The spectrum of the quadratic operator $\widetilde{q}^w(x,hD)$ on $H_{\Phi_{\delta,0}}(\C^n),$ for $0<\delta \ll1,$ is equal to the spectrum of $\widetilde{q}^w(x,hD)$ on $H_{\Phi_0}(\C^n),$ with the agreement of algebraic multiplies (see \cite{HitrikPravdaStarov2009}), and the above discussion may therefore be summarized in the following result.
    \begin{prop}\label{prop_int}
        Let $q$ be a quadratic form on $\R^{2n}$ such that $\re q \geq 0$ and (\ref{ass_dyn_1}) holds. Set $\widetilde{q} = q \circ \kappa_T^{-1}$ and define $\Phi_{\delta,0}$ as in (\ref{eq:phi_quad_defn}) for $\delta>0$ sufficiently small but fixed. Let $K \subseteq \C$ be a compact set disjoint from the spectrum of $q^w(x,D)$ on $L^2(\R^n).$ Then, for $0<h \leq \epsilon \leq 1,$ for all $\lambda \in K$ and all $u \in H(\C^n),$ one has
        \begin{equation}\label{eq:prop_localized_interior_estimate_final}
\begin{split}
    \|(h+|x|^2)^{\frac12} \chi_0\li(\frac{x}{\sqrt{\epsilon}}\ri)u\|_{L^2_{\Phi_{\delta,0}}} & \leq O(1) \|(h+|x|^2)^{-\frac12}\chi_0\li(\frac{x}{\sqrt{\epsilon}}\ri)(\widetilde{q}^w(x,hD)-h\lambda)u\|_{L^2_{\Phi_{\delta,0}}}\\ & +O( h^{\frac12})\|1_R\li(\frac{x}{\sqrt{\epsilon}}\ri)u\|_{L^2_{\Phi_{\delta,0}}}.
\end{split}
\end{equation}
Here, $\chi_0 \in C_c^\infty(\C^n;[0,1])$ is such that $\chi_0=1$ in a neighborhood of $0$ and $R \subseteq \C^n$ is a fixed neighborhood of $\text{supp}(\nabla \chi_0).$
\end{prop}

\section{Full resolvent estimates and proof of Theorem \ref{thm1}}\label{glue_together}
The purpose of this section is to complete the proof of Theorem \ref{thm1} by combining the results of Proposition \ref{prop_ext} and Proposition \ref{prop_int}. Fix $C>0$ and let $\lambda \in \C$ be the spectral parameter such that $\lambda \in D(0,C)$ and $\lambda \not \in \Omega,$ where $\Omega \subseteq \C$ is an open neighborhood of $\text{Spec}_{L^2(\R^n)}(q^w(x,D)).$ Our starting point is the estimate (\ref{eq:prop_localized_interior_estimate_final}), which we apply with $\epsilon=Ah$ for $A \geq 1$ and $K = \Omega^c \cap \overline{D(0,C)}.$ We first notice that for each fixed $A \geq 1,$ the $L^2_{\Phi_{\delta,0}}(\C^n)$-norms in (\ref{eq:prop_localized_interior_estimate_final}) can be replaced by norms in the space $L^2_{\Phi_{\delta,\epsilon}}(\C^n),$ provided that $0<h \leq h_0,$ where $h_0=h_0(A)>0$ is small enough depending on $A.$ Indeed, as already observed in \cite[Section 3]{hitrik_starov_i}, this follows from the fact that 
\begin{equation}
    \Phi_{\delta,\epsilon}(x) = \Phi_{\delta,0}(x)+O(\delta|x|^3), \quad |x| \leq \frac{\epsilon^{\frac12}}{C},
\end{equation}
(see (\ref{eq:G_eps_G_zero_Taylor}), (\ref{eq:phi_g_taylor}), (\ref{eq:phi_eps_zero_sim})).

In the region $|x| \leq \frac{\epsilon^{\frac12}}{C},$ we therefore have
\begin{equation}
    e^{-O(1)A^{\frac32}h^{\frac12}}\leq e^{-\frac{\Phi_{\delta,\epsilon}}{h}} e^{\frac{\Phi_{\delta,0}}{h}} \leq e^{O(1)A^{\frac32}h^{\frac12}}. 
\end{equation}
It follows that when $0<h \leq h_0(A),$ one may replace the weight $\Phi_{\delta,0}$ in (\ref{eq:prop_localized_interior_estimate_final}) by the weight $\Phi_{\delta,\epsilon}$ at the cost of a constant multiplicative error. Therefore, in view of (\ref{eq:prop_localized_interior_estimate_final}) and the discussion above,
\begin{equation}\label{eq:est_int}
\begin{split}
     \|(h+|x|^2)^{\frac12} \chi_0\li(\frac{x}{\sqrt{Ah}}\ri)u\|_{L^2_{\Phi_{\delta,\epsilon}}} \leq O(1)& \|(h+|x|^2)^{-\frac12}\chi_0\li(\frac{x}{\sqrt{Ah}}\ri)(\widetilde{q}^w(x,hD)-h\lambda)u\|_{L^2_{\Phi_{\delta,\epsilon}}}\\&+O( h^{\frac12})\|1_R\li(\frac{x}{\sqrt{Ah}}\ri)u\|_{L^2_{\Phi_{\delta,\epsilon}}},
\end{split}
\end{equation}
where $u \in H(\C^n),$ the cutoff function $\chi_0 \in C_c^\infty(\C^n;[0,1])$ is such that $\chi_0=1$ near $0,$ $R \subseteq \C^n$ is a neighborhood of $\text{supp}(\nabla \chi_0),$ and $0<h \leq h_0(A).$

Let $\widetilde{p}=\widetilde{p}(x,\xi;h)$ be as in (\ref{eq:ext_symbol}). We would like to replace the quadratic differential operator $\widetilde{q}^w(x,hD)$ in (\ref{eq:est_int}) by the operator $\widetilde{p}^w(x,hD;h)-p_0(0)-hp_1(0),$ committing only a small error. Here, we recall that $p_0(0) \in i \R$ and $p_1(0) \in \C$ is the value of the subprincipal symbol at the critical point $0 \in \R^{2n}.$ To this end, we observe that the almost holomorphic extension $\widetilde{p}_0$ in (\ref{eq:fbi_p_asymptotic}) satisfies 
\begin{equation}
    \widetilde{p}_0(x,\xi) = p_0(0)+\widetilde{q}(x,\xi)+O((x,\xi)^3) , \quad \C^{2n} \ni(x,\xi) \to (0,0),
\end{equation}
in view of (\ref{ass_quad}).
It follows that 
\begin{equation}\label{eq:p_tilde_taylor_full}
    \widetilde{p}(x,\xi;h)=p_0(0)+hp_1(0)+\widetilde{q}(x,\xi) +O(1)((x,\xi)^3+h(x,\xi)+h^2), \quad \C^{2n} \ni (x,\xi) \to (0,0).
\end{equation}
Combining (\ref{eq:p_tilde_taylor_full}) with the fact that $\widetilde{p} \in C_b^\infty(\C^{2n}),$ we obtain that 
\begin{equation}\label{eq:p_tilde_taylor_comp}
    |\widetilde{p}(x,\xi;h)-p_0(0)-hp_1(0)-\widetilde{q}(x,\xi)| \leq O(1)(|x|^3+|\xi|^3+h|x|+h|\xi|+h^2), \quad (x,\xi) \in \C^{2n}.
\end{equation}
Let $u \in T(\s(\R^n)),$ i.e. $u \in H(\C^n)$ and $u(x) = O_{k,h}(1) \la x \ra^{-k} e^{\frac{\Phi_0(x)}{h}}$ for all $k \in \N.$ When comparing the operators $\widetilde{p}^w(x,hD;h)-p_0(0)-hp_1(0)$ and $\widetilde{q}^w(x,hD)$ acting on $u,$ we may realize them both with the help of the contour given in (\ref{eq:deformed_contour_defn}), where $c>0$ is sufficiently large. We get
\begin{equation}\label{eq:p_q_kernel}
\begin{split}
        (\widetilde{p}^w(x,hD;h)& -p_0(0)-hp_1(0)-\widetilde{q}^w(x,hD))u(x)\\ &  = \frac{1}{(2\pi h)^n} \int_{\widetilde{\Gamma}_1(x)} e^{\frac{i}{h}(x-y) \cdot \theta} (\widetilde{p}-p_0(0)-hp_1(0)-\widetilde{q})\li(\frac{x+y}{2},\theta\ri)u(y)dy \wedge d\theta,
\end{split}
\end{equation}
where $\widetilde{\Gamma}_1(x)$ is given by 
\begin{equation}\label{eq:phi_contour_full}
    \theta = \frac2i \partial_x \Phi_{\delta,\epsilon}\li(\frac{x+y}{2}\ri)+ic(\overline{x-y}), \quad y \in \C^n.
\end{equation}
Writing 
\begin{equation}
     (\widetilde{p}^w(x,hD;h)-p_0(0)-hp_1(0)-\widetilde{q}^w(x,hD))u(x)=\int k_{\widetilde{\Gamma}_1}(x,y;h)u(y)d\mu(y),
\end{equation}
we obtain in view of (\ref{eq:p_tilde_taylor_comp}), (\ref{eq:p_q_kernel}), (\ref{eq:phi_contour_full}) that the absolute value of the effective kernel
\begin{equation}
    e^{-\frac{\Phi_{\delta,\epsilon}(x)}{h}} k_{\widetilde{\Gamma}_1}(x,y;h)e^{\frac{\Phi_{\delta,\epsilon}(y)}{h}}
\end{equation}
does not exceed, for some $F>0$ independent of $A,$
\begin{equation}\label{eq:kernel_estimate}
\begin{split}
    \frac{O(1)}{h^n} & e^{-\frac{F}{h}|x-y|^2}(|x|^3+|y|^3+|x-y|^3+\epsilon^{\frac32}+h|x|+h|y|+h|x-y|+h\epsilon^{\frac12}) \\ &
    \leq \frac{O_A(1)}{h^n}  e^{-\frac{F}{h}|x-y|^2}(|x|^3+|x-y|^3+h^{\frac32}+h|x|+h|x-y|).
\end{split}
\end{equation}
Here we have also used the first equality in (\ref{eq:phi_g_der_bounds}). Then, using (\ref{eq:kernel_estimate}) and Schur's lemma, we get
\begin{equation}\label{eq:kernel_quant_estimate}
    \|\chi_0\li(\frac{x}{\sqrt{Ah}}\ri)(\widetilde{p}^w(x,hD;h)-p_0(0)-hp_1(0)-\widetilde{q}^w(x,hD))u\|_{L^2_{\Phi_{\delta,\epsilon}}} =O_A(h^{\frac32})\|u\|_{L^2_{\Phi_{\delta,\epsilon}}},
\end{equation}
and combining (\ref{eq:est_int}) and (\ref{eq:kernel_quant_estimate}) leads to
\begin{equation}\label{eq:est_int_2}
\begin{split}
     \|(h+|x|^2)^{\frac12} &\chi_0\li(\frac{x}{\sqrt{Ah}}\ri)u\|_{L^2_{\Phi_{\delta,\epsilon}}} \\ & \leq O(1) \|(h+|x|^2)^{-\frac12}\chi_0\li(\frac{x}{\sqrt{Ah}}\ri)(\widetilde{p}^w(x,hD;h)-p_0(0)-hp_1(0)-h\lambda)u\|_{L^2_{\Phi_{\delta,\epsilon}}}\\&+O_A(h)\|u\|_{L^2_{\Phi_{\delta,\epsilon}}}+O( h^{\frac12})\|1_R\li(\frac{x}{\sqrt{Ah}}\ri)u\|_{L^2_{\Phi_{\delta,\epsilon}}}.
\end{split}
\end{equation}
Squaring (\ref{eq:est_int_2}), we obtain 
\begin{equation}\label{eq:est_int_3}
\begin{split}
     h\|\chi_0\li(\frac{x}{\sqrt{Ah}}\ri)u\|_{L^2_{\Phi_{\delta,\epsilon}}}^2 & \leq \frac{O(1)}{h} \|(\widetilde{p}^w(x,hD;h)-p_0(0)-hp_1(0)-h\lambda)u\|_{L^2_{\Phi_{\delta,\epsilon}}}^2\\&+O_A(h^2)\|u\|_{L^2_{\Phi_{\delta,\epsilon}}}^2+O( h)\|1_R\li(\frac{x}{\sqrt{Ah}}\ri)u\|_{L^2_{\Phi_{\delta,\epsilon}}}^2.
\end{split}
\end{equation}
Here, $u \in T(\s(\R^n)), 0<h \leq h_0(A),$ and $\lambda \in K.$

When estimating the last term on the right hand side of (\ref{eq:est_int_3}), we may use (\ref{eq:exterior_final_2}) in Proposition \ref{prop_ext}, which for $\re z \leq B$ gives 

\begin{equation}\label{eq:interior_full}
\begin{split}
    h\int_{\C^n}& 1_R\li(\frac{x}{\sqrt\epsilon}\ri)  |u(x)|^2e^{-\frac{2\Phi_{\delta,\epsilon}(x)}{h}}d\mu(x) \\ & \leq O(1) \|(\widetilde{p}^w(x,hD;h)-p_0(0)-hz)u\|_{{\Phi_{\delta,\epsilon}}}\|u\|_{\Phi_{\delta,\epsilon}}+\li(O_B\li(\frac{h}{A}\ri)+O_A(h^\infty)\ri)\|u\|_{\Phi_{\delta,\epsilon}}^2.
\end{split} 
\end{equation}
In our case, $z=p_1(0)+\lambda,$ where $\lambda \in D(0,C),$ and therefore we get 
\begin{equation}\label{eq:est_ext}
\begin{split}
    h\int_{\C^n}& 1_R\li(\frac{x}{\sqrt\epsilon}\ri)  |u(x)|^2e^{-\frac{2\Phi_{\delta,\epsilon}(x)}{h}}d\mu(x) \\ & \leq O(1) \|(\widetilde{p}^w(x,hD;h)-p_0(0)-hp_1(0)-h\lambda)u\|_{{\Phi_{\delta,\epsilon}}}\|u\|_{\Phi_{\delta,\epsilon}}+\li(O_C\li(\frac{h}{A}\ri)+O_A(h^\infty)\ri)\|u\|_{\Phi_{\delta,\epsilon}}^2
\end{split} 
\end{equation}
for all $\lambda \in D(0,C).$
By choosing the cutoffs $\chi$ and $\chi_0$ such that $\chi^2+\chi_0^2=1$ on $\C^n,$ we may combine (\ref{eq:exterior_final_2}), (\ref{eq:est_int_3}), and (\ref{eq:est_ext}) to obtain that 

\begin{equation}\label{eq:est_ext_2}
\begin{split}
    h \int_{\C^n} |u(x)|^2 &e^{-\frac{2\Phi_{\delta,\epsilon}(x)}{h}} d\mu(x) \leq O(1)\|(\widetilde{p}^w(x,hD;h)-p_0(0)-hp_1(0)-h\lambda)u\|_{\Phi_{\delta,\epsilon}}\|u\|_{\Phi_{\delta,\epsilon}} \\&+\frac{O(1)}{h}\|(\widetilde{p}^w(x,hD;h)-p_0(0)-hp_1(0)-h\lambda)u\|_{\Phi_{\delta,\epsilon}}^2 
     +\li(O_C\li(\frac{h}{A}\ri)+O_A(h^\infty)\ri)\|u\|_{\Phi_{\delta,\epsilon}}^2
\end{split}
\end{equation}
for all $u \in T(\s(\R^n)), \lambda \in K,$ and $0 <h \leq h_0(A).$ Using (\ref{eq:est_ext_2}) together with 
\begin{equation}
\begin{split}
O(1) \|(&\widetilde{p}^w(x,hD;h)-p_0(0)-hp_1(0)-h\lambda)u\|_{{\Phi_{\delta,\epsilon}}}\|u\|_{\Phi_{\delta,\epsilon}} \\ & \leq \frac{O(1)}{h} \|(\widetilde{p}^w(x,hD;h)-p_0(0)-hp_1(0)-h\lambda)u\|_{{\Phi_{\delta,\epsilon}}}^2+\frac{h}{2}\|u\|_{\Phi_{\delta,\epsilon}}^2,
    \end{split}
\end{equation}
we get 
\begin{equation}\label{eq:est_ext_3}
\begin{split}
h \int_{\C^n} |u(x)|^2 e^{-\frac{2\Phi_{\delta,\epsilon}(x)}{h}} d\mu(x) & \leq \frac{O(1)}{h} \|(\widetilde{p}^w(x,hD;h)-p_0(0)-hp_1(0)-h\lambda)u\|_{{\Phi_{\delta,\epsilon}}}^2\\ & +\li(O_C\li(\frac{h}{A}\ri)+ O_A(h^\infty)\ri)\|u\|_{\Phi_{\delta,\epsilon}}^2.
\end{split}
\end{equation}
For each $C>0,$ we may choose the parameter $A$ large enough and fixed so that 
\begin{equation}
    O_C\li(\frac{1}{A}\ri) \leq \frac{1}{2}.
\end{equation}
We then infer from (\ref{eq:est_ext_3}) that for all $u \in T(\s(\R^n))$ and all $h>0$ small enough, depending on the choice of $A,$ one has
\begin{equation}\label{eq:phi_g_simplified}
    h\|u\|_{\Phi_{\delta,\epsilon}} \leq O(1)\|(\widetilde{p}^w(x,hD;h)-p_0(0)-hp_1(0)-h\lambda)u\|_{\Phi_{\delta,\epsilon}}.
\end{equation}
It follows from (\ref{eq:phi_eps_zero_sim}) and (\ref{eq:phi_g_simplified}) that 
\begin{equation}\label{eq:phi_g_simplified_2}
    \|u\|_{\Phi_0} \leq O\li(\frac{1}{h}\ri)\|(\widetilde{p}^w(x,hD;h)-p_0(0)-hp_1(0)-h\lambda)u\|_{\Phi_0}.
\end{equation}
Using (\ref{eq:phi_g_simplified_2}), the unitarity of the FBI transform $T:L^2(\R^n) \to H_{\Phi_0}(\C^n),$ and (\ref{eq:egorov}), we get 
\begin{equation}\label{eq:resolvent_estimate_final}
   h \|u\|_{L^2(\R^n)} \leq O(1)\|(\widetilde{p}^w(x,hD;h)-p_0(0)-hp_1(0)-h\lambda)u\|_{L^2(\R^n)}
\end{equation}
for all $u \in \s(\R^n)$ and all $\lambda \in K.$ A density argument allows us to extend (\ref{eq:resolvent_estimate_final}) to all of $L^2(\R^n),$ implying that for all such $\lambda,$ the operator
\begin{equation}
    \widetilde{p}^w(x,hD;h)-p_0(0)-hp_1(0)-h\lambda: L^2(\R^n)\to L^2(\R^n)
\end{equation}
is injective with closed range. Recalling (\ref{ass_elliptic}), we observe that it is also Fredholm of index zero. It follows that 
\begin{equation}
     \widetilde{p}^w(x,hD;h)-p_0(0)-hp_1(0)-h\lambda: L^2(\R^n)\to L^2(\R^n), \quad \lambda \in K= \Omega^c \cap \overline{D(0,C)},
\end{equation}
is invertible for all $h>0$ small enough, with the inverse satisfying 
\begin{equation}
    \|(\widetilde{p}^w(x,hD;h)-p_0(0)-hp_1(0)-h\lambda)^{-1}\|_{\mathcal{L}(L^2(\R^n),L^2(\R^n))} \leq O_{C,\Omega}\li(\frac{1}{h}\ri).
\end{equation}
This completes the proof of Theorem \ref{thm1}. 
\appendix
\section{Sobolev spaces and Bargmann transforms}\label{appendix_a}
Let $m>0$ be an order function on $\R^{2n}$, in the sense that $m$ is Lebesgue measurable and there exist constants $C_0>0, N_0>0$ such that 
\begin{equation}
    m(X) \leq C_0 \la X-Y\ra^{N_0}m(Y), \quad X,Y \in \R^{2n}.
\end{equation}
Here, $\la X\ra=(1+|X|^2)^{\frac12}$ is the Japanese bracket of $X \in \R^{2n}.$ Associated to $m$ is the symbol class $S(\R^{2n},m)$
of all functions $a \in C^\infty(\R^{2n})$ such that 
\begin{equation}
    \frac{\partial^\alpha a}{m} \in L^\infty(\R^{2n}), \quad \forall\alpha \in 
    \N^{2n}.
\end{equation}
Regularizing $m$ without changing its order of magnitude, we may thus assume that $m \in S(\R^{2n},m)$ (see \cite[Lemma 8.7]{Zworski2012}). We can then introduce the Sobolev space\begin{equation}\label{weighted_sobolev_defn}
       H(m):=(m^w(x,hD))^{-1}(L^2(\R^n)) \subseteq \s'(\R^n),
\end{equation}
where $m^w(x,hD)$ is the $h$--Weyl quantization of $m$ and $(m^w(x,hD))^{-1}$ is its pseudodifferential inverse, defined for $h>0$ small enough \cite[Chapter 8]{Dimassi_Sjostrand_1999}. When equipped with the norm 
\begin{equation}
    \|u\|_{H(m)}=\|m^w(x,hD)u\|_{L^2(\R^n)},
\end{equation}
the space $H(m)$ becomes a Banach space and contains the Schwartz space $\s(\R^n)$ as a dense subspace.

Given a holomorphic quadratic form $\phi$ on $\C^{2n}$ with 
\begin{equation}
    \im \partial^2_{yy}\phi >0 , \quad \det \partial^2_{xy}\phi \neq 0,
\end{equation}
associated to $\phi$ is the complex linear canonical transformation
\begin{equation}
    \kappa: \C^{2n} \ni (y,-\partial_y \phi(x,y)) \to (x,\partial_x\phi(x,y)) \in \C^{2n},
\end{equation}
as well as the semiclassical metaplectic FBI-Bargmann transform 
\begin{equation}
    T u(x) = Ch^{-\frac{3n}{4}}\int_{\R^n} e^{\frac{i}{h}\phi(x,y)} u(y)dy,
\end{equation}
which defines a unitary map 
\begin{equation}\label{eq:egorov_unitary}
    T: L^2(\R^n) \to H_\Phi(\C^n)
\end{equation}
for a suitable choice of constant $C>0,$ see \cite[Section 1]{sjostrand_96}, \cite[Section 12.2]{sjostrand_resonance_lectures}, \cite[Section 1.3]{almost_holomorphic_ref}. Here,
\begin{equation}
    H_\Phi(\C^n) = H(\C^n) \cap L^2(\C^n,e^{-\frac{2\Phi}{h}}d\mu),
\end{equation}
where $\mu$ is the Lebesgue measure on $\C^n$ and $\Phi$ is a strictly plurisubharmonic quadratic form on $\C^n$ given by 
\begin{equation}
    \Phi(x) = \sup_{y \in \R^n} (-\im\phi(x,y)).
\end{equation}
We have furthermore, in view of \cite[Proposition 1.3.2]{almost_holomorphic_ref},
\begin{equation}
    \kappa(\R^{2n}) = \Lambda_\Phi:=\li\{\li(x,\frac2i\partial_x \Phi(x)\ri):x \in \C^n\ri\}.
\end{equation}
Let us also set
\begin{equation}\label{eq:order_fun_fbi}
    \widetilde{m}(x) := m\li(\kappa^{-1}\li(x,\frac{2}{i} \partial_x \Phi(x)\ri)\ri), \quad x \in \C^n, 
\end{equation}
and, in analogy with (\ref{weighted_sobolev_defn}), we define
 \begin{equation}
        H_{\Phi}(\widetilde{m}):= H(\C^n) \cap L^2(\C^n, \widetilde{m}^2 e^{-\frac{2\Phi}{h}}d\mu).
    \end{equation}

The purpose of this appendix is to demonstrate the following general, essentially well-known result:
   
\begin{prop}
    We have that for all $h>0$ small enough,
    \begin{equation}\label{eq:sobolev_surjective}
        T=O(1): H(m) \to H_{\Phi}(\widetilde{m})
    \end{equation}
    is an isomorphism with uniformly bounded inverse
    \begin{equation}\label{eq:sobolev_injective}
        T^{-1}=O(1): H_{\Phi}(\widetilde{m}) \to H(m).
    \end{equation}
\end{prop}
\begin{proof}
Let $a \in S(\R^{2n},m)$ and define $\widetilde{a} = a \circ \kappa^{-1} \in S(\Lambda_\Phi, m \circ \kappa^{-1}).$ Here, we may identify the order function $m \circ \kappa^{-1}$ on $\Lambda_\Phi$ with the order function $\widetilde{m}$ given by (\ref{eq:order_fun_fbi}). We recall the metaplectic invariance property
\begin{equation}\label{eq:egorov_2}
    Ta^w(x,hD)u = \widetilde{a}^w(x,hD)Tu, \quad u \in \s'(\R^n),
\end{equation}
see \cite[Proposition 1.4]{sjostrand_96}, where $\widetilde{a}^w: T(\s'(\R^n)) \to T(\s'(\R^n))$ is the $h$--Weyl quantization of $\widetilde{a}$ in the complex domain (see \cite[Section 1]{sjostrand_96} and also \cite{Hitrik_Johnson_2025}).

Now, let $e \in S(\R^{2n},\frac{1}{m})$ be such that for all $h>0$ small enough, one has
\begin{equation}
e^w(x,hD)m^w(x,hD)u=m^w(x,hD)e^w(x,hD)u=u, \quad u \in \s'(\R^n),
\end{equation}
see \cite[Chapter 8]{Dimassi_Sjostrand_1999}, and for $u \in H(m),$ using (\ref{eq:egorov_2}), let us write
\begin{equation}\label{eq:egorov_m_map}
    Tu=\widetilde{e}^w (T(m^wu)).
\end{equation}
Here, $T(m^w u) \in H_\Phi(\C^n)$ due to (\ref{weighted_sobolev_defn}) and (\ref{eq:egorov_unitary}), and $\widetilde{e} = e \circ \kappa^{-1} \in S(\Lambda_\Phi, \widetilde{m}^{-1}).$ An application of \cite[Proposition 12.6]{sjostrand_resonance_lectures} gives that 
\begin{equation}\label{eq:e_map}
    \widetilde{e}^w = O(1): H_\Phi(\C^n) \to H_{\Phi}(\widetilde{m}).
\end{equation}
Then, (\ref{eq:egorov_m_map}), (\ref{eq:e_map}) yield that $Tu \in H_\Phi(\widetilde{m})$ with
\begin{equation}
    \|Tu\|_{H_\Phi(\widetilde{m})} \leq O(1) \|T(m^wu)\|_{H_\Phi(\C^n)} \leq O(1) \|m^w u\|_{L^2(\R^n)} = O(1)\|u\|_{H(m)}.
\end{equation}
We therefore obtain (\ref{eq:sobolev_surjective}).
As it is generally known that $T: \s'(\R^n) \to H(\C^n)$ is injective, it remains to show that $T: H(m) \to H_{\Phi}(\widetilde{m})$ is surjective and that (\ref{eq:sobolev_injective}) holds. 
To this end, let $U \in H_{\Phi}(\widetilde{m}),$ and let us write $U=Tu$ for some unique $u \in \s'(\R^n)$ (see \cite[Proposition 6.1]{Hormander_1991_quadratic_hyperbolic}). Another application of \cite[Proposition 12.6]{sjostrand_resonance_lectures} then gives 
\begin{equation}
    \widetilde{m}^w = O(1): H_\Phi(\widetilde{m}) \to H_\Phi(\C^n),
\end{equation}
so that (\ref{eq:egorov_2}) implies 
\begin{equation}
    T(m^w u)= \widetilde{m}^w Tu \in H_\Phi(\C^n).
\end{equation}
We thus conclude that $u \in H(m)$ and 
\begin{equation}
    \|u\|_{H(m)} = \|m^w u\|_{L^2(\R^n)} = \|\widetilde{m}^wTu\|_{H_\Phi(\C^n)} \leq O(1) \|Tu\|_{H_{\Phi}(\widetilde{m})}.
\end{equation}
This completes the proof.
\end{proof}
\textit{Remark.} Using the quantization-multiplication formula \cite[Corollary 12.12]{sjostrand_resonance_lectures}, we see that (\ref{eq:sobolev_surjective}) and (\ref{eq:sobolev_injective}) can be sharpened to 
\begin{equation}
    \|u\|_{H(m)}=\|Tu\|_{H_{\Phi}(\widetilde{m})}(1+O(h)), \quad u \in H(m).
\end{equation}
\bibliographystyle{alpha}
\bibliography{refs_harm_approx_v2}
\end{document}